\long\def\symbolfootnote[#1]#2{\begingroup\def\thefootnote{\fnsymbol{footnote}}\footnote[#1]{#2}\endgroup}

\documentclass{amsart}
\usepackage{amssymb}
\usepackage{amsmath}
\usepackage{amsfonts}

\newtheorem{theorem}{Theorem}[section]
\newtheorem{corollary}[theorem]{Corollary}
\newtheorem{question}[theorem]{Question}
\newtheorem{lemma}[theorem]{Lemma}
\theoremstyle{remark}

\theoremstyle{definition}

\theoremstyle{proposition}
\newtheorem{proposition}[theorem]{Proposition}
\numberwithin{equation}{section}

\begin{document}
\title{Geometry of manifolds with densities}
\author{Ovidiu Munteanu and Jiaping Wang}

\begin{abstract}
We study geometry of complete Riemannian manifolds endowed with a weighted measure, where the weight
function is of quadratic growth. Assuming the associated Bakry-\'{E}mery curvature is bounded from below,
we derive a new Laplacian comparison theorem and establish various sharp volume upper and lower bounds.
We also obtain some splitting type results by analyzing the Busemann functions. In particular,
we show that a complete manifold with nonnegative Bakry-\'{E}mery curvature must split off a line if
it is not connected at infinity and its weighted volume entropy is of maximal value among linear growth
weight functions.

While some of our results are even new for the gradient Ricci solitons, the novelty here is that only
a lower bound of the Bakry-\'{E}mery curvature is involved in our analysis.
\end{abstract}

\maketitle
\section{Introduction}

\symbolfootnote[0]{Research of the first author supported in part by NSF grant DMS-1005484;
the second author by NSF grant DMS-1105799}

On a smooth Riemannian manifold $\left( M,g\right),$ the well-known
Bochner formula takes the form

\begin{equation*}
\frac{1}{2}\,\Delta |\nabla u|^2=|\mathrm{Hess}(u)|^2+\langle \nabla u, \nabla \Delta u\rangle+
\mathrm{Ric}(\nabla u, \nabla u)
\end{equation*}%
for any smooth function $u$ on $M,$ where $\mathrm{Hess} (u)$ denotes the hessian of $u$ and $\mathrm{Ric}$
the Ricci curvature of $M.$ This formula may be viewed as the definition of the Ricci curvature. Indeed,
it serves as the starting point of much of the analysis involving the Ricci curvature.

There are various occasions that consideration of weighted measure of the form $e^{-f}dv$ on $M$ is natural or warranted for a smooth function $f$ on $M.$
Such $f$ is called the weight function throughout this paper.
With respect to the weighted measure, one has the corresponding weighted Dirichlet energy functional $E_f(u)=\int_M |\nabla u|^2\,e^{-f}\,dv.$
Just as that the Laplacian $\Delta$ is associated with the Dirichlet energy,
the Euler-Lagrange operator of $E_f(u)$ is the so-called weighted Laplacian
$\Delta_f$ given by
\begin{equation*}
\Delta_f u=\Delta u-\langle \nabla f, \nabla u\rangle.
\end{equation*}%
In terms of $\Delta_f,$ the Bochner formula can be rewritten into

\begin{equation*}
\frac{1}{2}\,\Delta_f |\nabla u|^2=|\mathrm{Hess}(u)|^2+\langle \nabla u, \nabla \Delta_f u\rangle+
\mathrm{Ric}_f\,(\nabla u, \nabla u),
\end{equation*}%
where
\begin{equation*}
\mathrm{Ric}_{f}:=\mathrm{Ric}+\mathrm{Hess}\left( f\right).
\end{equation*}%
The curvature quantity $\mathrm{Ric}_f$ is called the Bakry-\'{E}mery curvature \cite{BE} of the smooth metric measure space $(M,g,e^{-f}dv).$

In \cite{BE}, Bakry-\'{E}mery showed that if $\mathrm{Ric}_f\geq \frac{1}{2},$
then the following logarithmic Sobolev inequality holds.

\begin{equation*}
\int_M u^2\,\ln u^2\, e^{-f}\,dv\leq 4 \int_M |\nabla u|^2 \,e^{-f}\,dv
\end{equation*}%
for compactly supported smooth function $u$ on $M$ satisfying

\begin{equation*}
\int_M u^2\, e^{-f}\,dv=\int_M e^{-f}\,dv.
\end{equation*}%
Specializing to the Euclidean space $\mathbb{R}^{n}$ with $f(x)=\frac{1}{4}\left\vert x\right\vert
^{2},$ one then recovers the classical logarithmic Sobolev inequality due to L. Gross \cite{G}.

More generally, in recent work of Lott and Villani \cite{LV}, and Sturm
\cite{S1, S2},
a lower bound for the Bakry-\'{E}mery curvature is interpreted as a convexity measurement for certain entropy functional on the space of probability measures over $M$ with respect to the Wasserstein distance. Indeed, this point of view has enabled them to define ``Ricci curvatures" for more general metric measure spaces.

Our interest in the Bakry-\'{E}mery curvature is largely due to gradient Ricci solitons. A Riemannian manifold $(M,g)$ is called a gradient Ricci soliton if
$\mathrm {Ric}_f=\lambda\, g$ for some function $f$ and constant $\lambda.$
Obviously, this is a generalization of the concept of Einstein manifolds.
A Ricci soliton is called shrinking, steady, or expanding, according to whether
the constant $\lambda$ is positive, zero, or negative.
Customarily, one assumes $\lambda=\frac{1}{2}, 0,$ or $-\frac{1}{2}$ by scaling
the metric. An easy example of nontrivial gradient Ricci solitons is the Euclidean space $\mathbb{R}^{n}$ with
$f(x)=\frac{\lambda}{2}\left\vert x\right\vert ^{2}+\langle a, x\rangle+b.$
Gradient Ricci solitons are naturally associated with the Ricci flows, as suggested by the name. Indeed, they are simply the self-similar solutions of the Ricci flows up to a suitable scaling. As well-known, knowledge of gradient Ricci solitons is crucial in the understanding of singularity formations of the Ricci flows \cite{H, CLN}. We shall refer readers to \cite{Cao} for a recent survey on Ricci solitons.

The study of smooth manifolds with the Bakry-\'{E}mery curvature bounded below
has been very active in recent years. Much effort has been devoted to establishing
results parallel to the case the Ricci curvature is bounded below. This
has been largely successful when the weight function $f$ is assumed to be bounded. We refer to the work of Lott \cite{L}, Wei and Wylie \cite{WW}, and ourselves
\cite{MW, MW1} for further details. In fact, Lichenorewicz \cite{Lich} has already observed that the classical Cheeger-Gromoll splitting theorem continues to hold for manifolds with $\mathrm {Ric}_f\ge 0$ if the weight function $f$ is bounded. Various results concerning volume comparison and analysis of the Laplacian have also been established.

When the weight function $f$ is of linear growth, by working with the weighted Laplacian, we have obtained in \cite{MW, MW1}, among other things, some results concerning the topology at infinity of manifolds with a lower bound on the
Bakry-\'{E}mery curvature. As an application, we concluded that a nontrivial steady gradient Ricci soliton must be connected at infinity.

It is useful to regard Ricci solitons as canonical examples of smooth metric measure spaces with the Bakry-\'{E}mery curvature bounded below. However, as indicated
by the results of Cao and Zhou \cite{CZ} and ourselves \cite{MW1}, the weight function $f$ is necessarily of quadratic growth for both shrinking and expanding solitons. Therefore, it requires one to study the Bakry-\'{E}mery curvature with $f$ being quadratic in order for the results to be relevant to gradient Ricci solitons. This is what we attempt to do here.

Our first result concerns the volume growth of $\left(
M,g\right) .$ The important point here is that we are able to  estimate the growth of the Riemannian volume, not the weighted volume. In the following, we denote by $r\left( x\right) :=d\left( p,x\right) $ the distance from $x$ to a fixed point $p\in M.$

\begin{theorem}
\label{Vol1}Let $\left( M,g,e^{-f}dv\right) $ be a complete smooth metric
measure space of dimension $n$ such that $\mathrm{Ric}_{f}\geq \lambda.$
Assume that there exist $\alpha ,\beta \ge 0$ so that
\begin{equation*}
\ \left\vert \nabla f\right\vert \left( x\right) \leq \alpha r\left(
x\right) +\beta \text{ on }M.
\end{equation*}%

\begin{enumerate}
\item If $\alpha =\lambda,$ then there exist positive constants $C_1, C_2>0$
so that
\begin{equation*}
\mathrm{Vol}\left( B_{p}\left( R\right) \right) \leq C_1\,e^{C_2\,\sqrt{R}}
\text{ \
\ for all }R>0.
\end{equation*}

\item If $\alpha >\lambda,$ then there exists a positive constant $C>0$
so that
\begin{equation*}
\mathrm{Vol}\left( B_{p}\left( R\right) \right) \leq C\,e^{\sqrt{\left(
n-1\right) \left( \alpha -\lambda\right) }\,R}
\text{ \ \ for all }R>0.
\end{equation*}
\end{enumerate}
\end{theorem}

Some remarks are in order. First, we point out that $\alpha \geq \lambda$
always holds true. This is obvious if $\lambda\leq 0.$ If $\lambda>0,$
then an argument using the second variation formula of the arc-length
(see \cite{CZ}) quickly implies the weight function $f$ is of at
least quadratic growth with $\alpha\geq \lambda.$ Secondly, when
$\alpha=\lambda=0,$ the result is essentially due to Sesum and the first author
\cite{MS}. Thirdly, if we consider $M=\mathbb{H}^{n},$ the hyperbolic space
of constant sectional curvature $-1,$ and let $f\left( x\right) :=\frac{n-1+\lambda
}{2}r^2 \left( x\right),$  then $f_{ij}\geq n-1+\lambda $ and $\mathrm{Ric}+\mathrm{Hess}\left( f\right) \geq \lambda$ on $M.$ Since
$\alpha=n-1+\lambda$ in this case, our estimate implies
the volume growth is%
\begin{equation*}
\mathrm{Vol}\left( B_{p}\left( R\right) \right) \leq C\,e^{(n-1)\,R},
\end{equation*}
which is obviously sharp.

If we denote the bottom spectrum of the Laplacian $\Delta$ of manifold $M$
by $\lambda_1(M),$ then according to \cite{LW1}, the volume of $M$ satisfies
\begin{equation*}
\mathrm{Vol}\left( B_{p}\left( R\right) \right) \geq C\,e^{2\,\sqrt {\lambda_1(M)} \,R} \text{ \ \ for all }R>1.
\end{equation*}
Combining these two estimates, we have the following corollary.

\begin{corollary}
Let $\left( M,g,e^{-f}dv\right) $ be a complete smooth metric
measure space of dimension $n$ with $\mathrm{Ric}_{f}\geq \lambda.$
Assume that there exist $\alpha ,\beta \ge 0$ so that
\begin{equation*}
\ \left\vert \nabla f\right\vert \left( x\right) \leq \alpha r\left(
x\right) +\beta \text{ on }M.
\end{equation*}%
Then the bottom spectrum $\lambda_1(M)$ satisfies
\begin{equation*}
\lambda_1(M)\le \frac{(n-1)(\alpha-\lambda)}{4}.
\end{equation*}%
\end{corollary}

Again, one sees this upper bound is sharp and achieved as seen from the
previous example. In view of the work of \cite{LW1} and \cite{LW2}, where
they have studied the equality case of Cheng's estimate \cite{C}
$\lambda_1(M)\le \frac{(n-1)^2}{4}$ on an $n$-dimensional manifold $M$
with $\mathrm {Ric}\ge -(n-1),$ one wonders if a parallel result exists.

\begin{question}
Let $\left( M,g,e^{-f}dv\right) $ be a complete smooth metric
measure space of dimension $n\ge 3$ with $\mathrm{Ric}_{f}\geq \lambda.$
Assume that there exist $\alpha ,\beta \ge 0$ so that
\begin{equation*}
\ \left\vert \nabla f\right\vert \left( x\right) \leq \alpha r\left(
x\right) +\beta \text{ on }M
\end{equation*}%
and the bottom spectrum $\lambda_1(M)$ satisfies
\begin{equation*}
\lambda_1(M)= \frac{(n-1)(\alpha-\lambda)}{4}.
\end{equation*}%
Is it true that $M$ must be connected at infinity or topologically a cylinder?
\end{question}

 The following result shows that the upper bound of volume growth can be improved to polynomial under a slightly stronger condition on the weight function $f.$
The result is obviously sharp.

\begin{theorem}
\label{Vol}Let $\left( M,g,e^{-f}dv\right) $ be a complete smooth metric
measure space of dimension $n$. Assume that
\begin{equation}
\mathrm{Ric}_{f}\geq \frac{1}{2}\text{ \ and \ }\left\vert \nabla
f\right\vert ^{2}\leq f.  \label{C}
\end{equation}%
Then

\begin{enumerate}

\item
  $\left( M,g\right) $ has Euclidean volume growth, i.e.,
\begin{equation*}
\mathrm{Vol}\left( B_{p}\left( R\right) \right) \leq c\left( n\right)
e^{f\left( p\right) }R^{n}
\end{equation*}%
for $R>0$ and $p\in M$. The constant $c\left( n\right) $ depends only on
$n.$

\item
$\left( M,g\right) $ has at least linear volume growth, i.e., for any $%
p\in M$ there exists a constant $c_{0}>0$ so that
\begin{equation*}
\mathrm{Vol}\left( B_{p}\left( R\right) \right) \geq c_{0}\,R,
\end{equation*}%
for $R>1$. The constant $c_{0}$ depends on $n,f\left( p\right) $ and
$\mu :=\int_{M}e^{-f}dv<\infty .$
\end{enumerate}
\end{theorem}

Recall that $\left( M,g,f\right) $ is a gradient shrinking Ricci soliton if $\mathrm{Ric}_{f}=\frac{1}{2}.$
Any gradient shrinking soliton verifies the assumptions in Theorem \ref{Vol}.
Indeed, by Hamilton's identity \cite{H} we know that $\left\vert \nabla
f\right\vert ^{2}+S=f,$ where $S$ is the scalar curvature of $M.$
Together with a result of Chen \cite{Ch}, see also \cite{Cao}, that $S\geq 0$, this implies
that the condition (\ref{C}) is satisfied by any shrinking Ricci solitons.
Hence, Theorem \ref{Vol} recovers the upper bound in \cite{CZ} and the lower
bound in \cite{MW1} about the volume growth of gradient shrinking solitons.

Instead of relying on the soliton equations, to prove Theorem \ref{Vol}, we develop some new Laplacian comparison type results for the distance functions on $M.$ It turns out these results can also be applied to obtain some splitting theorems,
which consists of the second part of our paper. Our first result in this direction
concerns manifolds with non-negative Bakry-\'{E}mery curvature. As mentioned earlier, the first splitting theorem for such manifolds goes back to
Lichnerowicz \cite{Lich}, where $f$ is assumed to be bounded on $M.$  More recently,
in \cite{FLZ}, this assumption was relaxed to $f$ being bounded above. These are
natural generalizations of the celebrated Cheeger-Gromoll theorem \cite{CG}
for manifolds with non-negative Ricci curvature. In the following Theorem \ref%
{Ent}, we deal with the case that $f$ is not necessarily bounded.

\begin{theorem}
\label{Ent}Let $\left( M,g,e^{-f}dv\right) $ be a smooth metric measure
space of dimension $n$ with $\mathrm{Ric}_{f}\geq 0.$ Assume that $f$ has linear
growth rate $a>0$ and the weighted volume entropy of $\left(
M,g,e^{-f}dv\right) $ satisfies $h_{f}\left( M\right) =a.$ Then either $M$ is connected at infinity or $M$ splits as a direct product $M=\mathbb{R}\times N,$
where $N$ is compact.
\end{theorem}

Here, $f$ is said to grow linearly on $M$ if
\begin{equation*}
\left\vert f\right\vert \left( x\right) \leq \alpha r\left( x\right) +\beta
\end{equation*}%
for some constants $\alpha ,\beta >0.$ The infimum value of all such $\alpha $ is then called the linear growth rate of $f.$ The
weighted volume entropy of $\left( M,g,e^{-f}dv\right)$ is defined as%
\begin{equation}
h_{f}\left( M\right) :=\limsup_{r\rightarrow \infty }\frac{\ln \mathrm{Vol}%
_{f}\left( B_{p}\left( r\right) \right) }{r}.  \label{ve}
\end{equation}%
It is known if $\mathrm{Ric}_{f}\geq 0$ and $f$ has linear growth rate
$a,$ then (see \cite{MW})%
\begin{equation*}
h_{f}\left( M\right) \leq a.
\end{equation*}%
So our theorem deals with the case that $h_{f}\left( M\right)$ is maximal.

Note by \cite{B, LW1} the following inequality holds true in general.
\begin{equation}
\lambda _{1}\left( \Delta _{f}\right) \leq \frac{1}{4}h_{f}^{2}\left(
M\right),  \label{ineq}
\end{equation}%
where $\lambda _{1}\left( \Delta _{f}\right) $ is the bottom spectrum of the weighted Laplacian $\Delta_f.$
In particular, if $\lambda _{1}\left( \Delta _{f}\right)$ achieves its maximal value $\frac{a^2}{4},$ then $h_{f}\left( M\right)=a.$
Previously, we have proved the theorem in \cite{MW} under the stronger assumption
that $\lambda _{1}\left( \Delta _{f}\right) $ is maximal.

In the case $f$ is allowed to be of quadratic growth, we have the following result.

\begin{theorem}
\label{Split}Let $\left( M,g,e^{-f}dv\right) $ be a smooth metric measure
space of dimension $n$ with $\mathrm{Ric}_{f}\geq \frac{1}{2}.$ Assume
that $f$ satisfies the upper bound%
\begin{equation*}
f\left( x\right) \leq \frac{1}{4}d\left( x,K\right) ^{2}+C\text{ \ on }M
\end{equation*}%
for some constant $C>0$ and compact set $K.$ Then there is no line passing
through $K$ or $M$ is isometric to $\mathbb{R}\times N.$ In particular,
if $M\backslash K$ has at least two unbounded components, then $M=\mathbb{R}\times N$
for some compact manifold $N.$
\end{theorem}

For gradient shrinking Ricci solitons, it is known \cite{CZ} that $f$
satisfies
\begin{equation*}
\left( \frac{1}{2}r\left( x\right) -c\right) ^{2}\leq f\left( x\right) \leq
\left( \frac{1}{2}r\left( x\right) +c\right) ^{2}
\end{equation*}%
However, at this point, it is unclear to us how to apply our result to study
shrinking Ricci solitons.

The paper is organized as follows. In Section \ref{VU}, after establishing
some new Laplacian comparison results, we prove
the volume upper bound estimates in Theorem \ref{Vol1} and \ref{Vol}. We then prove the splitting result Theorem \ref{Ent} in Section \ref{lin} and Theorem \ref{Split} in Section \ref{qua}. In Section \ref{VL}, we prove the lower bound estimate in Theorem \ref{Vol}.

\section{\label{VU}Comparison geometry and volume upper bound}

In this section, we will first establish a version of Laplace comparison
theorem in terms of $\mathrm {Ric}_f$ lower bound. The result will then be used to
derive some sharp volume upper bounds. The comparison theorem will also be
invoked throughout the rest of the paper.

Let $\left( M,g,e^{-f}dv\right) $ be a smooth metric measure space with
\begin{equation*}
\mathrm{Ric}_{f}\geq \lambda,
\end{equation*}%
where $\lambda$ is a constant and usually normalized
by scaling the metric $g$ as $\lambda=\frac{1}{2},\ 0,\ -\frac{1}{2}.$
For a fixed point $p\in M,$ we denote the volume form in geodesic coordinates by
\begin{equation*}
dV|_{\exp _{p}\left( r\xi \right) }=J\left( p,r,\xi \right) drd\xi
\end{equation*}%
for $r>0$ and $\xi \in S_{p}M,$ the unit tangent sphere at $p.$ It is well
known that if $x\in M$ is any point outside the cut locus of $p,$ such that $%
x=\exp _{p}$ $\left( r\xi \right) ,$ then
\begin{equation*}
\Delta d\left( p,x\right) =\frac{d}{dr}\ln J\left( p,r,\xi \right) .
\end{equation*}%
In the following, we will omit the dependence of these quantities on $p$ and
$\xi .$ Along a minimizing geodesic $\gamma $ starting from $p$, we have
(see e.g. \cite{L1})
\begin{equation}
m^{\prime }\left( r\right) +\frac{1}{n-1}m^{2}\left( r\right) +\mathrm{Ric}%
\left( \frac{\partial }{\partial r},\frac{\partial }{\partial r}\right) \leq
0,  \label{vu1}
\end{equation}%
where the differentiation is with respect to the $r$ variable and we have
denoted $m\left( r\right) :=\frac{d}{dr}\ln J\left( r\right) .$ As
in \cite{WW}, multiplying (\ref{vu1}%
) by $r^{2}$ and integrating from $r=0$ to $r=t>0,$ we get%
\begin{equation}
\int_{0}^{t}m^{\prime }\left( r\right) r^{2}dr+\frac{1}{n-1}%
\int_{0}^{t}m^{2}\left( r\right) r^{2}dr+\lambda %
\int_{0}^{t}r^{2}dr\leq \int_{0}^{t}f^{\prime \prime }\left( r\right)
r^{2}dr,  \label{vu2}
\end{equation}%
where we have used $\mathrm{Ric}_{f}\geq \lambda.$ Here,
\begin{equation*}
f^{\prime \prime }\left( r\right) :=\mathrm{Hess}_{f}\left( \frac{\partial }{%
\partial r},\frac{\partial }{\partial r}\right) =\frac{d^{2}}{dr^{2}}\left(
f\circ \gamma \right) \left( r\right) .
\end{equation*}%
Integrating the first and the last terms in (\ref{vu2}) by parts and
rearranging terms, we arrive at%
\begin{gather*}
m\left( t\right) t^{2}+\frac{1}{n-1}\int_{0}^{t}\left( m\left( r\right)
r-\left( n-1\right) \right) ^{2}dr\leq \left( n-1\right) t-\frac{\lambda }{3}%
t^{3} \\
+t^{2}f^{\prime }\left( t\right) -2\int_{0}^{t}f^{\prime }\left( r\right)
rdr.
\end{gather*}%
In particular, after discarding the square term on the left hand side, we
obtain as in \cite{FLZ, WW}%
\begin{equation}
m\left( t\right) \leq \frac{n-1}{t}-\frac{\lambda }{3}t+f^{\prime }\left(
t\right) -\frac{2}{t^{2}}\int_{0}^{t}sf^{\prime }\left( s\right) ds.
\label{vu3}
\end{equation}%
Formula (\ref{vu3}) has a drawback that it involves $\nabla f$ twice with
no obvious cancelations. Moreover, the last term requires information of $\nabla f$
on the whole geodesic $\gamma.$ We now perform the following computation to resolve these points. First, integrate (\ref{vu3}) from $t=\varepsilon $ to $%
t=r>\varepsilon $ for $\varepsilon >0$ small. Then
\begin{gather}
\ln J\left( r\right) -\ln J\left( \varepsilon \right) \leq \left( n-1\right)
\left( \ln r-\ln \varepsilon \right) -\frac{\lambda }{6}\left(
r^{2}-\varepsilon ^{2}\right)  \label{vu4} \\
+f\left( r\right) -f\left( \varepsilon \right) -\int_{\varepsilon }^{r}\frac{%
2}{t^{2}}\left( \int_{0}^{t}sf^{\prime }\left( s\right) ds\right) dt.  \notag
\end{gather}%
Integrating by parts implies %
\begin{equation*}
-\int_{\varepsilon }^{r}\frac{2}{t^{2}}\left( \int_{0}^{t}sf^{\prime }\left(
s\right) ds\right) dt=\frac{2}{t}\left( \int_{0}^{t}sf^{\prime }\left(
s\right) ds\right) \left\vert _{t=\varepsilon }^{t=r}\right.
-2\int_{\varepsilon }^{r}f^{\prime }\left( t\right) dt
\end{equation*}%
Plugging this into (\ref{vu4}) and letting $\varepsilon \rightarrow 0,$
we conclude%
\begin{equation*}
\ln J\left( r\right) \leq \left( n-1\right) \ln r-\frac{\lambda }{6}%
r^{2}-\left( f\left( r\right) -f\left( 0\right) \right) +\frac{2}{r}%
\int_{0}^{r}sf^{\prime }\left( s\right) ds.
\end{equation*}%
Rearranging terms, we have%
\begin{equation*}
-\frac{2}{r^{2}}\int_{0}^{r}sf^{\prime }\left( s\right) ds\leq -\frac{1}{r}%
\ln \left( \frac{J\left( r\right) }{r^{n-1}}\right) -\frac{\lambda }{6}r-%
\frac{1}{r}\left( f\left( r\right) -f\left( 0\right) \right).
\end{equation*}%
After plugging into (\ref{vu3}), one obtains
\begin{equation}
m\left( r\right) \leq \frac{n-1}{r}-\frac{\lambda }{2}r+f^{\prime }\left(
r\right) -\frac{1}{r}\ln \left( \frac{J\left( r\right) }{r^{n-1}}\right) -%
\frac{1}{r}\left( f\left( r\right) -f\left( 0\right) \right) .  \label{vu5}
\end{equation}%
This holds for any $0\leq r\leq R$. The importance of (\ref{vu5}) lies in
the the fact that it now involves $f$ only at the two endpoints $p$ and $x$ of the geodesic $\gamma .$

Obviously, we can rewrite (\ref{vu5}) into %
\begin{equation}
\Delta _{f}d\left( p,x\right) \leq \frac{n-1}{r}-\frac{\lambda }{2}r-\frac{1%
}{r}\ln \left( \frac{J\left( p,r,\xi \right) }{r^{n-1}}\right) -\frac{1}{r}%
\left( f\left( x\right) -f\left( p\right) \right) ,  \label{vu6}
\end{equation}%
where $x=\exp _{p}\left( r\xi \right) .$ Let us stress that (\ref{vu5}) and (\ref{vu6}) use
only $\mathrm{Ric}_{f}\geq \lambda.$

To summarize, we have

\begin{proposition} Let $\gamma$ be a minimizing normal geodesic in $M$ with
$p=\gamma (0)$ and $x=\gamma (r).$ If $\mathrm{Ric}_{f}\geq \lambda,$
then

\begin{equation}
\Delta d\left( p,x\right)\leq \frac{n-1}{r}-\frac{\lambda }{2}r+f^{\prime }\left(
r\right) -\frac{1}{r}\ln \left( \frac{J\left( r\right) }{r^{n-1}}\right) -%
\frac{1}{r}\left( f\left( r\right) -f\left( 0\right) \right)
\end{equation}%
and
\begin{equation}
\Delta _{f}d\left( p,x\right) \leq \frac{n-1}{r}-\frac{\lambda }{2}r-\frac{1%
}{r}\ln \left( \frac{J\left( p,r,\xi \right) }{r^{n-1}}\right) -\frac{1}{r}%
\left( f\left( x\right) -f\left( p\right) \right).
\end{equation}%
\end{proposition}

We are now ready to prove the first part of Theorem \ref{Vol} which is restated
below. The lower bound estimate of the volume will be handled in section 5.

\begin{theorem}
\label{Vol_Up}Let $\left( M,g,e^{-f}dv\right) $ be a complete smooth metric
measure space of dimension $n$. Assume that $\mathrm{Ric}_{f}\geq \frac{1}{2}
$ and $\left\vert \nabla f\right\vert ^{2}\leq f.$ Then $\left( M,g\right) $
has at most Euclidean volume growth, i.e.
\begin{equation*}
\mathrm{Vol}\left( B_{p}\left( R\right) \right) \leq C_{0}R^{n}
\end{equation*}%
for any $R>0$ and $p\in M,$ where the constant $C_{0}$ can be chosen as $%
C_{0}=c\left( n\right) e^{f\left( p\right) }.$
\end{theorem}

\begin{proof}[Proof of Theorem \protect\ref{Vol_Up}]
For fixed $p\in M,$  according to (\ref{vu5}), we have:%
\begin{equation}
m\left( r\right) \leq \frac{n-1}{r}-\frac{1}{4}r+f^{\prime }\left( r\right) -%
\frac{1}{r}\ln \left( \frac{J\left( r\right) }{r^{n-1}}\right) -\frac{1}{r}%
\left( f\left( r\right) -f\left( 0\right) \right) .  \label{vu9}
\end{equation}%
Since $\left\vert \nabla f\right\vert ^{2}\leq f,$ we have:%
\begin{equation*}
f^{\prime }\left( r\right) \leq \frac{1}{4}r+\frac{1}{r}\left\vert \nabla
f\right\vert ^{2}\leq \frac{1}{4}r+\frac{1}{r}f(r).
\end{equation*}%
Plugging this into (\ref{vu9}) leads to
\begin{equation}
m\left( r\right) -\frac{n-1}{r}+\frac{1}{r}\ln \left( \frac{J\left( r\right)
}{r^{n-1}}\right) \leq \frac{f\left( p\right) }{r}.  \label{vu10}
\end{equation}%
Since $m\left( r\right) :=\frac{d}{dr}\ln J\left( r\right) ,$ we see that (%
\ref{vu10}) becomes
\begin{equation}
\left( rH\left( r\right) \right) ^{\prime }\leq f\left( p\right) ,
\label{vu11}
\end{equation}%
where $H\left( r\right) :=\ln \frac{J\left( r\right) }{r^{n-1}}.$
Integrating (\ref{vu11}) from $r=0$ to $r=R$ implies%
\begin{equation*}
J\left( R\right) \leq e^{f\left( p\right) }R^{n-1} \ \ \text{for all }R>0.
\end{equation*}%
This proves an area estimate by integrating on $S_{p}M.$ The claimed volume
estimate follows immediately.
\end{proof}

Let us point out that for shrinking Ricci solitons, which satisfy the
equation $\mathrm{Ric}_{f}=\frac{1}{2},$ one has $\left\vert \nabla f\right\vert ^{2}\leq f.$
Also, if $p$ is chosen to be a minimum point of $f,$ then
$f\left( p\right) \leq \frac{n}{2}$ and the constant $C_{0}$ in Theorem \ref%
{Vol_Up} depends only on $n$ as shown in \cite{HM}.

We now prove Theorem \ref{Vol1}. For this, we will follow the
argument in Lemma 2.1 of \cite{MW1}. First, let us point out that if $\left(
M,g,e^{-f}dv\right) $ has $\mathrm{Ric}_{f}\geq \frac{1}{2}$ and there exist
$\alpha ,\beta >0$ so that
\begin{equation*}
\ \left\vert \nabla f\right\vert \left( x\right) \leq \alpha r\left(
x\right) +\beta \text{ on }M,
\end{equation*}%
then $\alpha \geq \frac{1}{2}.$ This follows by a standard argument using
the second variation formula of arc-length as in the proof of Proposition \ref{f}
below. More exactly, this follows immediately from (\ref%
{f1}).

\begin{theorem}
\label{VolU}Let $\left( M,g,e^{-f}dv\right) $ be a complete smooth metric
measure space of dimension $n$ such that $\mathrm{Ric}_{f}\geq \lambda.$
Assume that there exist $\alpha ,\beta \geq 0$ so that
\begin{equation*}
\ \left\vert \nabla f\right\vert \left( x\right) \leq \alpha r\left(
x\right) +\beta \text{ on }M.
\end{equation*}%
(a) If $\alpha =\lambda,$ then there exist positive constants $C_1, C_2>0$ so
that
\begin{equation*}
\mathrm{Vol}\left( B_{p}\left( R\right) \right) \leq C_1\,e^{C_2\,\sqrt{R}} \text{ \
\ for }R>0.
\end{equation*}%
(b) If $\alpha >\lambda,$ then there exists a positive constant $C>0$ so
that
\begin{equation*}
\mathrm{Vol}\left( B_{p}\left( R\right) \right) \leq Ce^{\sqrt{\left(
n-1\right) \left( \alpha -\lambda\right) }\,R} \text{ \ \ for }R>0.
\end{equation*}
\end{theorem}

\begin{proof}[Proof of Theorem \protect\ref{VolU}]
As before, we write $dV|_{\exp _{p}\left( r\xi \right) }=J\left(
p,r,\xi \right) drd\xi $ for $\xi \in S_{p}M.$ We integrate (\ref{vu1}) from
$1$ to $r\geq 1$ and get
\begin{equation*}
m\left( r\right) +\frac{1}{n-1}\int_{1}^{r}\left( m\left( t\right) \right)
^{2}dt\leq f^{\prime }\left( r\right) -\lambda\,(r-1)+m\left( 1\right)
-f^{\prime }\left( 1\right),
\end{equation*}%
where $m\left( r\right) :=\frac{d}{dr}\ln J\left( p,r,\xi \right) $. Using the hypothesis on the growth of $\left\vert \nabla f\right\vert $
and the Cauchy-Schwarz inequality, we conclude that
\begin{equation}
m\left( r\right) +\frac{1}{\left( n-1\right) r}\left( \int_{1}^{r}m\left(
t\right) dt\right) ^{2}\leq \left( \alpha -\lambda\right) r+C_{0},
\label{v1}
\end{equation}%
where $C_{0}>0$ is a constant so that $m\left( 1\right) -f^{\prime
}\left( 1\right) +\lambda +\beta \leq C_{0}$. Let us assume from now on
that $\alpha >\lambda.$

We claim that for any $r\geq 1,$
\begin{equation}
\int_{1}^{r}m\left( t\right) dt\leq \sqrt{\left( n-1\right) \left( \alpha -%
\lambda\right) }\, r+C,  \label{v2}
\end{equation}%
where
\begin{equation*}
C:=\sqrt{\frac{n-1}{\alpha -\lambda}}C_{0}.
\end{equation*}%
To prove this, define
\begin{equation*}
v\left( r\right) :=\sqrt{\left( n-1\right) \left( \alpha -\lambda\right)
}\, r+C-\int_{1}^{r}m\left( t\right) dt.
\end{equation*}%
We show instead that $v\left( r\right) \geq 0$ for all $r\geq 1.$ Clearly, $%
v\left( 1\right) >0.$ Suppose that $v$ does not remain positive for all $%
r\geq 1.$ Let $R>1$ be the first number such that $v\left( R\right) =0.$
By the choice of $R$ we know that $v^{\prime }\left( R\right) \leq 0$.
So we have
\begin{eqnarray*}
\int_{1}^{R}m\left( t\right) dt &=&\sqrt{\left( n-1\right) \left( \alpha -%
\lambda\right) }\,R+C\text{ \ \ and} \\
m\left( R\right) &\geq &\sqrt{\left( n-1\right) \left( \alpha -\lambda%
\right) }.
\end{eqnarray*}%
It is then easy to see that%
\begin{eqnarray*}
\frac{1}{\left( n-1\right) R}\left( \int_{1}^{R}m\left( t\right) dt\right)
^{2} &\geq &\left( \alpha -\lambda\right) R+2C\sqrt{\frac{\alpha -%
\lambda}{n-1}} \\
&>&\left( \alpha -\lambda\right)\,R+C_{0}.
\end{eqnarray*}%
Since $m\left( R\right) >0,$ this is a contradiction to (\ref{v1}). We
have thus proved that (\ref{v2}) is true for any $r\geq 1.$ In other words,
\begin{equation}
\ln J\left( r\right) \leq \sqrt{\left( n-1\right) \left( \alpha -\lambda%
\right) }\, r+\sqrt{\frac{n-1}{\alpha -\lambda}}C_{0}+\ln J\left(
1\right) .  \label{v3}
\end{equation}%
Recall that the constant $C_{0}$ is independent of both $r$ and $\alpha $.
Now (\ref{v3}) clearly implies the conclusion of the
theorem for all $\alpha >\lambda.$

If $\alpha =\lambda,$ we apply (\ref{v3}%
) to $\alpha ^{\prime }:=\lambda+\frac{1}{r}>\lambda,$ which shows
that
\begin{equation*}
\ln J\left( r\right) \leq A\sqrt{r}
\end{equation*}%
for all $r\geq 1.$ The desired result then follows.
\end{proof}

We point out that it is possible to prove Theorem \ref{VolU} by a
similar argument as in Theorem \ref{Vol_Up} by using $\sinh^{2}\sqrt{\frac{\alpha -\lambda}{n-1}}\,r$ instead of $r^{2}$ in (\ref{vu2}). However, the computations now are much more involved. One advantage of doing so is that one also obtains relative volume comparison estimates as well.

As a corollary, we obtain an upper bound for the bottom spectrum of the Laplacian
$\Delta$ on $M.$

\begin{corollary}
Let $\left( M,g,e^{-f}dv\right) $ be a complete smooth metric
measure space of dimension $n$ such that $\mathrm{Ric}_{f}\geq \lambda.$
Assume that there exist $\alpha ,\beta >0$ so that
\begin{equation*}
\ \left\vert \nabla f\right\vert \left( x\right) \leq \alpha r\left(
x\right) +\beta \text{ on }M.
\end{equation*}%
Then the bottom spectrum $\lambda_1(M)$ of the Laplacian
$\Delta$ on $M$ satisfies

\begin{equation*}
\lambda_1(M)\le \frac{(n-1)(\alpha-\lambda)}{4}.
\end{equation*}%
\end{corollary}

As pointed out in section 1, both upper bounds for volume growth and the bottom
spectrum are sharp as demonstrated by the hyperbolic spaces.

\section{\label{lin}Splitting for linear weight}

In this section we apply the Laplace comparison theorem established in
Section \ref{VU} to prove a splitting theorem for manifolds with non-negative Bakry-\'{E}mery curvature,
where the weight function is assumed of linear growth.
So let $\left( M,g,e^{-f}dv\right) $ be a smooth metric
measure space with $\mathrm{Ric}_{f}\geq 0.$ We furthermore assume that $f$
satisfies
\begin{equation*}
\left\vert f\right\vert \left( x\right) \leq \alpha r\left( x\right) +\beta ,
\end{equation*}
where $r\left( x\right) :=d\left( p,x\right) $ is the distance to a fixed
point $p\in M,$ and $\alpha ,\beta >0$ are positive constants. The infimum value
of all such $\alpha $ is called to be the linear growth rate of $f.$ For a proof
of the following result, see Theorem 2.1 in \cite{MW} or the argument below
in Theorem \ref{MaxEntropy}.

\begin{theorem}
\label{Entropy}Let $\left( M,g,e^{-f}dv\right) $ be a smooth metric measure
space with $\mathrm{Ric}_{f}\geq 0.$ Assume that $f$ has linear growth rate
$a\geq 0.$
Then the weighted volume entropy $h_f(M)$ of $\left( M,g,e^{-f}dv\right)$
satisfies %
\begin{equation*}
h_{f}\left( M\right) :=\limsup_{r\rightarrow \infty }\frac{\ln \mathrm{Vol}%
_{f}\left( B_{p}\left( r\right) \right) }{r}\le a.
\end{equation*}
\end{theorem}

The estimate for $h_{f}\left( M\right) $ is in fact sharp for any gradient
steady Ricci soliton by \cite{MS}. Our goal here is to understand the structure of
$M$ when this volume entropy is maximal. Recall that an end of $M$ is any unbounded
component of $M\backslash K$, where $K\subset M$ is a compact smooth domain.
Manifold $M$ is said to be connected at infinity if it has only one end with respect to any such $K.$

We have the following result.

\begin{theorem}
\label{MaxEntropy}Let $\left( M,g,e^{-f}dv\right) $ be a smooth metric
measure space of dimension $n$ with $\mathrm{Ric}_{f}\geq 0.$ Assume that $f$
has linear growth rate $a>0$ and the weighted volume entropy of $%
\left( M,g,e^{-f}dv\right) $ is maximal, that is, $h_{f}\left( M\right) =a.$ Then
either $M$ is connected at infinity or $M$ splits as a direct product $M=%
\mathbb{R}\times N,$ where $N$ is compact.
\end{theorem}

\begin{proof}[Proof of Theorem \protect\ref{MaxEntropy}]
We assume that $M$ has at least two ends and prove that M splits off a line as claimed in the theorem.

For a fixed $\varepsilon >0$ and point $p\in M,$  we first show that there exists a sequence of
points $x_{k}\in M$ which goes to $\infty $ and
\begin{equation}
\frac{f\left( x_{k}\right) }{d\left( p,x_{k}\right) }\leq -a+\varepsilon
\text{ \ for all }k.  \label{e1}
\end{equation}%
Indeed, if (\ref{e1}) is not true, then
\begin{equation}
\frac{f\left( x\right) }{d\left( p,x\right) }>-a+\varepsilon   \label{e2}
\end{equation}%
for all $x$ such that $r\left( x\right) \geq r_{0}$ for some $r_{0}>0.$
Let us denote the volume form in geodesic coordinates centered at $p$ by
\begin{equation*}
dV|_{\exp _{p}\left( r\xi \right) }=J\left( p,r,\xi \right) drd\xi
\end{equation*}%
for $r>0$ and $\xi \in S_{p}M,$ the unit tangent sphere at $p$. For $R>r_{0},$
we consider a normal minimizing geodesic $%
\gamma \left( t\right) $  from $p=\gamma \left(0\right) $ to
$x=\gamma \left( R\right).$  According to (\ref%
{vu5}), we have%
\begin{equation*}
m\left( r\right) \leq \frac{n-1}{r}+f^{\prime }\left( r\right) -\frac{1}{r}%
\ln \left( \frac{J\left( r\right) }{r^{n-1}}\right) -\frac{1}{r}\left(
f\left( r\right) -f\left( 0\right) \right) ,
\end{equation*}%
where $m\left( r\right) :=\frac{d}{dr}\ln J\left( p,r,\gamma ^{\prime
}\left( 0\right) \right) .$ Denote $u\left( r\right) :=\ln \frac{%
J_{f}\left( r\right) }{r^{n-1}}.$ The above inequality becomes%
\begin{equation*}
u^{\prime }\left( r\right) +\frac{1}{r}u\left( r\right) \leq \frac{f\left(
0\right) }{r}-2\frac{f\left( r\right) }{r}.
\end{equation*}%
Choosing $r>r_{\varepsilon }$ sufficiently large, by (\ref{e2}) we see that%
\begin{equation*}
u^{\prime }\left( r\right) +\frac{1}{r}u\left( r\right) \leq 2a-\varepsilon ,%
\text{ \ for all }r>r_{\varepsilon }.
\end{equation*}%
Integrating this inequality from $r=r_{\varepsilon }$ to $r=R>r_{\varepsilon
},$ we conclude that
\begin{equation*}
J_{f}\left( R\right) =R^{n-1}e^{u\left( R\right) }\leq C\left( \varepsilon
\right) e^{\left( a-\frac{1}{2}\varepsilon \right) R}.
\end{equation*}%
Since $R$ is arbitrary, this implies that the weighted volume entropy is $h_{f}\left(
M\right) \leq a-\frac{1}{2}\varepsilon$, which is a contradiction.
Therefore, (\ref{e1}) is true.

Let $E$ denote one of the ends of $M.$ Let $F:=M\backslash E$. Then $%
M=E\cup F$ and $\mathrm{Int}\left( E\right) \cap \mathrm{Int}\left( F\right)
=\phi .$ Moreover, $E$ and $F$ have compact boundary, and $\partial
E=\partial F$.

By choosing a subsequence if necessary, we may assume that $%
x_{k}\in E$ for all $k\geq 0.$ By construction of $E$ and $F$ it is easy
to see that there exists a ray from $x_{k}$ to the infinity of $F$ for each $%
k.$ Let $y_{k}$ be a point on this ray such that the mid point of the geodesic
segment $\overline{x_{k}y_{k}}$ lies on $\partial E.$ We denote by $\gamma
_{k}:=\overline{x_{k}y_{k}}$ and parametrize $\gamma _{k}$ such that
\begin{equation}
\gamma _{k}\left( -t_{k}\right) =x_{k}\in E\text{, \ \ }\gamma _{k}\left(
t_{k}\right) =y_{k}\in F\text{\ \ and \ }\gamma _{k}\left( 0\right)
=z_{k}\in \partial E.  \label{e3}
\end{equation}%
Since the sequence $x_{k}$ is unbounded, $t_{k}\rightarrow
\infty $ as $k\rightarrow \infty .$ We now consider the following functions%
\begin{eqnarray*}
\beta _{k}^{+}\left( x\right) &:=& t_{k}-d\left( y_{k},x\right)  \\
\beta _{k}^{-}\left( x\right) &:=& t_{k}-d\left( x_{k},x\right)  \\
\beta _{k}\left( x\right) &:=& \beta _{k}^{+}\left( x\right) +\beta
_{k}^{-}\left( x\right)
\end{eqnarray*}%
and show that $\beta _{k}^{\pm }$ satisfies a differential
inequality in the sense of distributions. Moreover, $\beta _{k}^{\pm }$
converges uniformly on compact sets. To this end, let $\Omega \subset M$ be an arbitrary compact
 domain and $\phi \in C^{\infty }\left( M\right) $ a smooth nonnegative function with support in $\Omega .$

Let $x\in \Omega $ be any point which is not in the cut locus of either $%
x_{k}$ or $y_{k}.$ Denote by $\tau _{k}\left( s\right) $ the unique
minimizing geodesic from $y_{k}$ to $x,$ parametrized by arclength. We let $r_{k}:=d\left( y_{k},x\right).$
Then we have $\tau
_{k}\left( 0\right) =y_{k}$ and $\tau _{k}\left( r_{k}\right) =x.$ According
to (\ref{vu6}), we have
\begin{equation}
\Delta _{f}d\left( y_{k},x\right) +\frac{1}{r_{k}}\ln J\left(
y_{k},r_{k},\tau _{k}^{\prime }\left( 0\right) \right) \leq \frac{n-1}{r_{k}}%
\left( 1+\ln r_{k}\right) -\frac{f\left( x\right) }{r_{k}}+\frac{f\left(
y_{k}\right) }{r_{k}}.  \label{e4}
\end{equation}%
Moreover, since the linear growth rate of $f$ is $a,$ for any $\varepsilon >0$ we
have
\begin{equation*}
\left\vert f\left( y_{k}\right) \right\vert \leq \left( a+\varepsilon
\right) r_{k}+b.
\end{equation*}%
Note that $b$ is a positive constant depending only on $\Omega $ and $%
\varepsilon ,$ but not $k.$ Hence, for $k\geq
k_{\varepsilon }$ sufficiently large, it follows that:
\begin{equation}
\Delta _{f}d\left( y_{k},x\right) +\frac{1}{r_{k}}\ln J\left( r_{k}\right)
\leq a+2\varepsilon .  \label{e5}
\end{equation}%
This inequality holds for any $x\in \Omega$ not in the cut locus
of $y_{k},$ and for all $k\geq k_{\varepsilon }$ with $k_{\varepsilon }$
depending on $\varepsilon $ and $\Omega .$
Recall that $r_{k}\left( x\right) :=d\left( y_{k},x\right) .$ Since $\Omega $ is
bounded, there exists a constant $c>0$ independent of $k$ and $%
\varepsilon $ so that
\begin{equation*}
\Omega \subset B_{y_{k}}\left( t_{k}+c\right) \backslash B_{y_{k}}\left(
t_{k}-c\right) .
\end{equation*}%
 Moreover, there exists a constant $c_{0}>0$, independent of $k$ and $\varepsilon$, so that $%
r_{k}>t_{k}-c_{0}$, for any $x\in \Omega .$

We now multiply (%
\ref{e5}) by
\begin{equation*}
\phi dv=\phi J\left( y_{k},r_{k},\xi \right) dr_{k}d\xi ,
\end{equation*}%
where $\phi $ is a nonnegative function with support in $\Omega $. Since the function $%
h\left( J\right) :=J\ln J$ is bounded below by $-\frac{1}{e},$ one has
\begin{gather*}
\int_{\mathbb{S}^{n-1}}J\left( y_{k},r_{k},\xi \right) \ln J\left(
y_{k},r_{k},\xi \right) \phi d\xi \geq -\frac{1}{e}\int_{\mathbb{S}%
^{n-1}}\phi d\xi \\
\geq -C\sup_{\Omega }\phi .
\end{gather*}%
In view of this, it then follows from (\ref{e5}) that
\begin{equation*}
\int_{M}\left( \Delta _{f}d\left( y_{k},x\right) -a\right) \phi \leq C\left(
\frac{1}{t_{k}-c_{0}}+2\varepsilon \right) \sup_{\Omega }\phi,
\end{equation*}%
where the constant $C$ is independent of $k$ and $%
\varepsilon.$ So for $k_{\varepsilon }>0$ sufficiently large,
\begin{equation}
\int_{M}\left( \Delta _{f}\beta _{k}^{+}+a\right) \phi \geq -C\varepsilon
\sup_{\Omega }\phi   \label{e6}
\end{equation}%
for any $k\geq k_{\varepsilon }$ and any nonnegative function $\phi $
with support on $\Omega .$

In the same fashion, we obtain a similar differential inequality for $\beta
_{k}^{-}.$ Indeed, using the same notations as before, we have:%
\begin{equation*}
\Delta _{f}d\left( x_{k},x\right) +\frac{1}{\overline{r}_{k}}\ln J\left(
\overline{r}_{k}\right) \leq \frac{n-1}{\overline{r}_{k}}\left( 1+\ln
\overline{r}_{k}\right) -\frac{f\left( x\right) }{\overline{r}_{k}}+\frac{%
f\left( x_{k}\right) }{\overline{r}_{k}},
\end{equation*}%
where $\overline{r}_{k}:=d\left( x_{k},x\right) .$ Recall
that by (\ref{e1}) $x_{k}$ satisfies
\begin{equation*}
\frac{f\left( x_{k}\right) }{d\left( p,x_{k}\right) }\leq -a+\varepsilon .
\end{equation*}%
Therefore, the same argument as for $\beta _{k}^{+}$ leads to
\begin{equation}
\int_{M}\left( \Delta _{f}\beta _{k}^{-}-a\right) \phi \geq -C\varepsilon
\sup_{\Omega }\phi   \label{e7}
\end{equation}%
for any $k\geq k_{\varepsilon }$ and any nonnegative function $\phi $
with support on $\Omega .$ The constant $C$ in (\ref{e7}) is again independent of $%
k$ and $\varepsilon .$

We now show that the sequences $\beta _{k}^{+}$ and $\beta _{k}^{-}$ admit
subsequences which converge uniformly on compact sets of $M$. Indeed, by the construction,
both functions are uniformly Lipschitz. Moreover,
\begin{equation*}
\beta _{k}^{\pm }\left( z_{k}\right) =\beta _{k}^{\pm }\left( \gamma
_{k}\left( 0\right) \right) =0\text{.}
\end{equation*}%
Note that the sequence $z_{k}$ is bounded as $z_{k}\in
\partial E.$ This shows that both $\beta _{k}^{+}$ and $\beta _{k}^{-}$ are
uniformly bounded and equicontinuous on any compact set. So we can apply Arzela-Ascoli theorem to
conclude a subsequence of $\beta _{k}^{\pm }$
converges locally uniformly on $M.$ We denote the limits by
$\beta ^{+}$ and $\beta ^{-},$ respectively.

Without loss of generality, we may assume $z_{k}=\gamma _{k}\left( 0\right) \in \partial E $ converges
to $z\in \partial E.$ Moreover, we may assume,
by further passing to a subsequence, that $\gamma _{k}\rightarrow \gamma$
uniformly on compact sets. The geodesic $\gamma $ obtained in this fashion
is in fact a line with $\gamma \left( 0\right) =z.$ Note that
\begin{equation*}
\beta _{k}^{+}\left( \gamma _{k}\left( t\right) \right) =t\text{ \ and \ }%
\beta _{k}^{-}\left( \gamma _{k}\left( t\right) \right) =-t
\end{equation*}%
for all $-t_{k}\leq t\leq t_{k}$. It follows that
\begin{equation}
\beta ^{+}\left( \gamma \left( t\right) \right) =t\text{ \ and \ }\beta
^{-}\left( \gamma \left( t\right) \right) =-t  \label{e8}
\end{equation}%
for all $t\in \mathbb{R}$. Hence, we know that $\beta ^{\pm }$ are linear
along $\gamma.$  Since $\left\vert \nabla\beta ^{\pm }\right\vert \leq 1$ on $M,$
we conclude that $\left\vert \nabla\beta ^{\pm }\right\vert =1$ along $\gamma.$

After letting $k\rightarrow \infty $ in (\ref{e6}), we have:%
\begin{equation*}
\int_{M}\left( \Delta _{f}\beta ^{+}+a\right) \phi \geq -C\varepsilon
\sup_{\Omega }\phi
\end{equation*}%
for any nonnegative smooth function $\phi $ with support in $\Omega.$
Since $C$ is independent of $%
\varepsilon,$ by setting $\varepsilon\rightarrow 0,$ one concludes that%
\begin{equation*}
\int_{M}\left( \Delta _{f}\beta ^{+}+a\right) \phi \geq 0
\end{equation*}%
for any nonnegative smooth function $\phi $ with compact support. In other words,
\begin{equation}
\Delta _{f}\beta ^{+}\geq -a  \label{e9}
\end{equation}%
in the sense of distributions. Similarly, one concludes that
\begin{equation}
\Delta _{f}\beta ^{-}\geq a  \label{e10}
\end{equation}%
in the sense of distributions. So for $\beta :=\beta ^{+}+\beta ^{-},$ it
follows that $\Delta _{f}\beta \geq 0$ on $M.$ Notice, however, that $\beta \left(
\gamma \left( t\right) \right) =0$ by (\ref{e8}). Also, $\beta \leq 0$
on $M$ as $\beta _{k}^{+}+\beta _{k}^{-}\leq 0$ for any $k$ by the
triangle inequality. Therefore, $\beta $ achieves its maximum along $\gamma.$
 The strong maximum principle then implies that $\beta =0$ on $M.$

From (\ref{e9}) and (\ref{e10}), it is then easy to see that
\begin{equation}
\Delta _{f}\beta ^{+}=-a\text{ on }M.  \label{e11}
\end{equation}%
In particular, $\beta ^{+}$ is smooth on $M.$ We now show that $%
\left\vert \nabla \beta ^{+}\right\vert =1$ on $M.$ Clearly, by the definition,
we have $\left\vert \nabla \beta ^{+}\right\vert \leq 1$ on $M.$
Furthermore, (\ref{e8}) shows that $\left\vert \nabla \beta ^{+}\right\vert
=1$ along $\gamma .$ Hence, $\left\vert \nabla \beta ^{+}\right\vert $
achieves its maximum on $\gamma .$ Now by (\ref{e11}) and the Bochner
formula, we get
\begin{align}
\frac{1}{2}\Delta _{f}\left\vert \nabla \beta ^{+}\right\vert ^{2}& \geq
\left\vert \mathrm{Hess}\left( \beta ^{+}\right) \right\vert ^{2}+\mathrm{Ric%
}_{f}\left( \nabla \beta ^{+},\nabla \beta ^{+}\right)  \label{e12} \\
+\left\langle \nabla \Delta _{f}\beta ^{+},\nabla \beta ^{+}\right\rangle &
\geq 0.  \notag
\end{align}%
Hence, $\left\vert \nabla \beta ^{+}\right\vert ^{2}$ is $f-$subharmonic.
By the maximum principle, it must be constant. This proves that $\left\vert
\nabla \beta ^{+}\right\vert =1$ on $M.$ From (\ref{e12}), it also follows
that $\mathrm{Hess}\left( \beta ^{+}\right) =0$ on $M.$ Therefore, the vector field
$\nabla \beta ^{+}$ is nontrivial and parallel. In particular, it implies $M$
splits into $M=\mathbb{R}\times N$ with $N$ being a level set of $\beta^{+}.$
Since $\beta ^{+}\left( t,x\right)=t$ for all $x\in N,$ (\ref{e11}) implies that $f^{\prime }\left(t\right) =a,$
or $f\left( t,x\right) =at+h\left( x\right) ,$ where $h\in C^{\infty }\left( N\right) .$ Since $N$ is totally
geodesic, we have
$\mathrm{\ Ric}_{N}+\mathrm{Hess}%
\left( h\right) \geq 0.$ In conclusion, $\left( N,g,e^{-h}dv\right) $ is a compact
smooth metric measure space with \textrm{Ric}$_{h}\geq 0.$ This proves the
theorem.
\end{proof}

The preceding proof shows that if $h_f(M)=a,$ then by picking a sequence
$x_k\in M$ going to infinity with
\begin{equation*}
\lim_{k\to \infty} \frac{f\left( x_{k}\right) }{d\left( p,x_{k}\right) }\leq -a,
\end{equation*}%
the associated Busemann function $\beta$ to the geodesic ray obtained as
the limit of minimizing geodesics from any fixed point $p$ to $x_k$ satisfies
\begin{equation}
\Delta _{f}\beta \geq a.  \label{e13}
\end{equation}%

\begin{theorem}
\label{lowerbd}
Let $\left( M,g,e^{-f}dv\right) $ be a smooth metric
measure space of dimension $n$ with $\mathrm{Ric}_{f}\geq 0$ and $|f|(x)\le a\,r(x)+b$ for some positive constants $a$ and $b.$
Assume that the weighted volume entropy of $%
\left( M,g,e^{-f}dv\right) $ satisfies $h_{f}\left( M\right) =a.$ Then
$\mathrm{Vol}(B_p(R))\ge c\,R$ for some constant $c>0$ and all $R\ge 1.$
\end{theorem}

\begin{proof}[Proof of Theorem \protect\ref{lowerbd}]
From the preceding remark, there exists a geodesic ray emanating from point $p$ such that the associated Busemann function $\beta$ satisfies the differential inequality
\begin{equation*}
\Delta _{f}\beta \geq a.
\end{equation*}%
Integrating the inequality with respect to the weighted measure on geodesic ball
$B_p(R),$ one concludes that
\begin{equation*}
a\, \mathrm{Vol}_f(B_p(R))\le \int_{B_p(R)}\Delta _{f}\beta \, e^{-f}dv\leq \mathrm{Area}_f(\partial B_p(R)),
\end{equation*}%
where $\mathrm{Vol}_f(B_p(R))$ and $\mathrm{Area}_f(\partial B_p(R))$ are the weighted volume and area of the respective sets. Thus, we have
\begin{equation*}
\mathrm{Area}_f(\partial B_p(R))\geq c\,e^{a\,R}
\end{equation*}%
for $R\ge 1.$

On the other hand, as $|f|(x)\le a\,r(x)+b,$ one sees easily
that
\begin{equation*}
\mathrm{Area}_f(\partial B_p(R))\leq e^b\,e^{a\,R}\,\mathrm{Area}(\partial B_p(R)).
\end{equation*}%
Combining these two inequalities, one concludes that
\begin{equation*}
\mathrm{Area}(\partial B_p(R))\geq c
\end{equation*}%
 for some positive constant $c.$ The theorem then follows.
\end{proof}

\section{\label{qua}Splitting for quadratic weight}

In this section we consider smooth metric measure spaces with
the Bakry-\'{E}mery curvature bounded below by a positive constant,
which is normalized as $\mathrm{Ric}%
_{f}\geq \frac{1}{2}.$ Our goal is to study the structure at infinity of
such manifolds. A new difficulty here is that $f$ grows at
least quadratically on $M.$

\begin{theorem}
\label{Splitting}Let $\left( M,g,e^{-f}dv\right) $ be a smooth metric
measure space of dimension $n$ with $\mathrm{Ric}_{f}\geq \frac{1}{2}.$
Assume that $f$ admits an estimate%
\begin{equation}
f\left( x\right) \leq \frac{1}{4}d\left( x,K\right) ^{2}+C,  \label{s*}
\end{equation}%
for all $x\in M,$ where $C>0$ is some constant and $K\subset M$ a compact subset.
If there is a line of $M$ passing through $K,$ then $M=\mathbb{R}\times N.$  In particular, if $M\backslash K$ has more than
one unbounded component, then there exists a compact manifold $N$ such that $%
M=\mathbb{R}\times N.$
\end{theorem}

\begin{proof}[Proof of Theorem \protect\ref{Splitting}]
It suffices to deal with the first case that there is a line $\gamma$
passing through $K.$
Indeed, if $M\backslash K$ is not connected at infinity, denoting $E$ to be one
of the unbounded connected components of $M\backslash K$ and $F:=\left( M\backslash
K\right) \backslash E,$ then it follows that
\begin{equation*}
E\cup F=M\backslash K\text{ \ \ and }E\cap F=\phi .
\end{equation*}%
Also, $F$ is an unbounded component of $M$. Note that any point in $E$ can be
connected via a minimizing geodesic to another point in $F.$ Such a geodesic
must intersect $K$ at least once. It is then standard to construct a
line $\gamma \left( t\right) ,$ $t\in \mathbb{R}$, such that $\gamma \left(
0\right) \in K,$ $\gamma \left( t\right) \in E$ for $t<0$ sufficiently
negative and $\gamma \left( t\right) \in F$ for $t>0$ sufficiently large.

By (\ref{s*}), we have
\begin{equation}
f\left( \gamma \left( t\right) \right) \leq \frac{1}{4}t^{2}+C,\text{ \ \
for all }t\in \mathbb{R}.  \label{s0}
\end{equation}
For $t>0$ and a point $x\in M$ not in the cut locus of $\gamma \left( t\right),$ denote by $\tau _{t}\left( s\right) $ the unique minimizing normal geodesic from
$\gamma \left( t\right) $ to $x.$ Let $r:=d\left( \gamma \left( t\right) ,x\right).$ Then $\tau_{t}\left( 0\right) =\gamma \left( t\right) $ and $\tau _{t}\left( r\right)=x.$ According to (\ref{vu6}), we have%
\begin{gather}
\Delta _{f}d\left( \gamma \left( t\right) ,x\right) +\frac{1}{r}\ln J\left(
\gamma \left( t\right) ,r,\tau _{t}^{\prime }\left( 0\right) \right) \leq
\frac{n-1}{r}\left( 1+\ln r\right)  \label{s3} \\
-\frac{1}{4}r+\frac{f\left( \gamma \left( t\right) \right) }{r}-\frac{%
f\left( x\right) }{r}.  \notag
\end{gather}
We now claim that for any $\varepsilon >0$ and $t>0$ sufficiently large,
\begin{equation}
\frac{n-1}{r}\left( 1+\ln r\right) -\frac{1}{4}r+\frac{f\left( \gamma \left(
t\right) \right) }{r}-\frac{f\left( x\right) }{r}\leq \frac{1}{2}\left(
t-r\right) +\varepsilon .  \label{s4}
\end{equation}%
To verify this, first note that by the triangle inequality,
\begin{equation}
\left\vert t-r\right\vert \leq d\left( \gamma \left( 0\right) ,x\right) .
\label{s5}
\end{equation}%
Since $x$ is fixed, for $t>0$ sufficiently large we have
\begin{equation*}
\frac{n-1}{r}\left( 1+\ln r\right) -\frac{f\left( x\right) }{r}\leq \frac{%
\varepsilon }{2}.
\end{equation*}%
Using (\ref{s0}), we have,%
\begin{eqnarray*}
&&-\frac{1}{4}r+\frac{f\left( \gamma \left( t\right) \right) }{r}-\frac{1}{2}%
\left( t-r\right) =\frac{1}{4r}\left( -r^{2}+4f\left( \gamma \left( t\right)
\right) -2r\left( t-r\right) \right) \\
&\leq &\frac{1}{4r}\left( -r^{2}+t^{2}+C-2r\left( t-r\right) \right) =\frac{1%
}{4r}\left( \left( t-r\right) ^{2}+C\right) \\
&\leq &\frac{1}{4r}\left( d\left( p,x\right) ^{2}+C\right),
\end{eqnarray*}%
where we have used (\ref{s5}) in the last step. Hence, for $t>0$
sufficiently large,
\begin{equation*}
-\frac{1}{4}r+\frac{f\left( \gamma \left( t\right) \right) }{r}-\frac{1}{2}%
\left( t-r\right) \leq \frac{\varepsilon }{2}.
\end{equation*}%
In conclusion, (\ref{s4}) holds true. It then follows from (\ref{s3})
that, for $t>0$ sufficiently large,
\begin{equation}
\Delta _{f}d\left( \gamma \left( t\right) ,x\right) +\frac{1}{r}\ln J\left(
\gamma \left( t\right) ,r,\xi \right) \leq \frac{1}{2}\left( t-d\left(
\gamma \left( t\right) ,x\right) \right) +\varepsilon ,  \label{s7}
\end{equation}%
where $\xi :=\tau _{t}^{\prime }\left( 0\right) ,$ and $r:=d\left( \gamma
\left( t\right) ,x\right) .$

Let us emphasize that (\ref{s7}) holds for any $x$ not in
the cut-locus of $\gamma \left( t\right).$ For a compact domain
 $\Omega \subset M$ and a nonnegative smooth function $\phi$ with
 support in $\Omega,$ using the
geodesic coordinates centered at $\gamma \left( t\right),$ we multiply (\ref%
{s7}) by
\begin{equation*}
\phi dv=\phi J\left( \gamma \left( t\right) ,r,\xi \right) drd\xi
\end{equation*}%
and integrate over $\Omega.$

Since $\Omega $ is bounded, there exists a constant $c>0$ independent of $t$
so that
\begin{equation*}
\Omega \subset B_{\gamma \left( t\right) }\left( t+c\right) \backslash
B_{\gamma \left( t\right) }\left( t-c\right) .
\end{equation*}%
Consequently, there exists a constant $c_{0}>0$ so that $r>t-c_{0}$ whenever
$x\in \Omega .$ Furthermore, since the function $h\left( J\right) :=J\ln J$
is bounded below by $-\frac{1}{e},$ we see that
\begin{gather*}
\int_{\mathbb{S}^{n-1}}J\left( \gamma \left( t\right) ,r,\xi \right) \ln
J\left( \gamma \left( t\right) ,r,\xi \right) \phi d\xi \geq -\frac{1}{e}%
\int_{\mathbb{S}^{n-1}}\phi d\xi \\
\geq -C\sup_{\Omega }\phi .
\end{gather*}%
Using this and (\ref{s7}), we conclude that
\begin{gather}
\int_{M}\left( \Delta _{f}d\left( \gamma \left( t\right) ,x\right) \right)
\phi \leq \frac{C}{t-c_{0}}\sup_{\Omega }\phi  \label{s8} \\
+\int_{M}\left( \frac{1}{2}\left( t-d\left( \gamma \left( t\right) ,x\right)
\right) \right) \phi +\varepsilon \sup_{\Omega }\phi .  \notag
\end{gather}%
We now denote the Busemann function associated with the ray $\gamma \left(
t\right) ,$ $t>0,$ by
\begin{equation*}
\beta ^{+}\left( x\right) :=\lim_{t\rightarrow \infty }\left( t-d\left(
\gamma \left( t\right) ,x\right) \right) .
\end{equation*}%
Letting $t\rightarrow \infty $ in (\ref{s8}) and then $\varepsilon
\rightarrow 0,$ we have
\begin{equation*}
\int_{M}\left( \Delta _{f}\beta ^{+}+\frac{1}{2}\beta ^{+}\right) \phi \geq
0.
\end{equation*}%
This shows that
\begin{equation}
\Delta _{f}\beta ^{+}\geq -\frac{1}{2}\beta ^{+}\text{ \ \ on }M  \label{s9}
\end{equation}%
in the sense of distributions. Similarly, for
\begin{equation*}
\beta ^{-}\left( x\right) :=\lim_{t\rightarrow \infty }\left( t-d\left(
\gamma \left( -t\right) ,x\right) \right) ,
\end{equation*}%
the same argument as above implies
\begin{equation}
\Delta _{f}\beta ^{-}\geq -\frac{1}{2}\beta ^{-}\text{ \ \ on }M
\label{s10}
\end{equation}%
in the sense of distributions.

Adding (\ref{s9}) to (\ref{s10}) and denoting $\beta :=\beta ^{+}+\beta
^{-}, $ we obtain
\begin{equation}
\Delta _{f}\beta \geq -\frac{1}{2}\beta .  \label{s11}
\end{equation}%
Notice now that $\beta \in L^{2}\left( e^{-f}dv\right) .$ Indeed, the
assumption that $\mathrm{Ric}_{f}\geq \frac{1}{2}$ implies (see \cite{M, WW})
that %

\begin{equation*}
\mathrm{Area}_{f}\left( \partial B_{p}\left( t\right) \right) \leq ce^{-%
\frac{1}{4}t^{2}+ct},
\end{equation*}%
which suffices to justify our claim.

Integrating (\ref{s11}) on $(M,e^{-f}dv)$ and using that $\beta \leq 0$ on $M,$ one concludes $\beta =0$ on $M.$ Hence, $%
\beta ^{+}=-\beta ^{-},$ and both (\ref{s9}) and (\ref{s10}) have to be
equalities. In particular, $\beta ^{+}$ is an eigenfunction of $\Delta _{f}$
with eigenvalue $\frac{1}{2}$. By the Bochner formula, we have
\begin{eqnarray*}
0 &=&\frac{1}{2}\Delta _{f}\left\vert \nabla \beta ^{+}\right\vert ^{2} \\
&=&\left\vert \mathrm{Hess}\left( \beta ^{+}\right) \right\vert ^{2}+\mathrm{%
Ric}_{f}\left( \nabla \beta ^{+},\nabla \beta ^{+}\right) +\left\langle
\nabla \Delta _{f}\beta ^{+},\nabla \beta ^{+}\right\rangle \\
&\geq &\left\vert \mathrm{Hess}\left( \beta ^{+}\right) \right\vert ^{2}.
\end{eqnarray*}%
This proves that $\mathrm{Hess}\left( \beta ^{+}\right) =0.$ Since $%
\left\vert \nabla \beta ^{+}\right\vert =1,$ this implies the splitting $M=%
\mathbb{R}\times N,$ $N$ being a level set of $%
\beta ^{+},$ i.e., $N=\left( \beta ^{+}\right) ^{-1}\left\{ 0\right\} .$
Clearly, $\beta ^{+}\left( t,x\right) =t$ for all $x\in N$ and $t\in
\mathbb{R}.$ Moreover, from $\Delta _{f}\beta ^{+}=-\frac{1}{2}\beta ^{+},$
we find that $\left\langle \nabla f,\nabla t\right\rangle =\frac{1}{2}t$
or $f\left( t,x\right) =\frac{1}{4}t^{2}+h\left( x\right) ,$ where $h\in
C^{\infty }\left( N\right) .$ Since $\mathrm{Ric}_{f}\geq \frac{1}{2}%
,$  it implies $\mathrm{Ric}_{N}+\mathrm{Hess}\left(
h\right) \geq \frac{1}{2}$ on $N$ as well. The theorem is
proved.
\end{proof}

We now point out the following estimate on weight function $f.$
The argument is essentially due to Cao and Zhou \cite{CZ}, where they
have obtained the same result for shrinking gradient Ricci solitons.

\begin{proposition}
\label{f} Let $\left( M,g,e^{-f}dv\right) $ be a complete smooth metric
measure space of dimension $n$. Assume that $\mathrm{Ric}_{f}\geq \frac{1}{2}
$ and $\left\vert \nabla f\right\vert ^{2}\leq f.$ Then there exists a
constant $a>0$ so that%
\begin{equation*}
\frac{1}{4}\left( d\left( p,x\right) -a\right) ^{2}\leq f\left( x\right)
\leq \frac{1}{4}\left( d\left( p,x\right) +a\right) ^{2},
\end{equation*}%
for any $x\in M$ with $d(p,x)\geq r_{0}$. The constants $a$ and $r_{0}$
depend only on $n$ and $f\left( p\right) .$
\end{proposition}

\begin{proof}[Proof of Proposition \protect\ref{f}]
We follow the argument in \cite{CZ} for gradient shrinkers closely.
Rewrite $\left\vert \nabla f\right\vert ^{2}\leq f$ as
\begin{equation}
\left\vert \nabla \left( 2\sqrt{f}\right) \right\vert \leq 1.  \label{f0}
\end{equation}%
Integrating along a minimizing geodesics gives
\begin{equation*}
\sqrt{f}\left( x\right) \leq \frac{1}{2}d\left( p,x\right) +\sqrt{f\left(
p\right) }.
\end{equation*}%
This proves the upper bound estimate. We now prove the lower bound.
Note by the second variation formula of arc-length,
along a minimizing geodesic $\gamma $ from $p$ to $x,$ we have
\begin{equation*}
\int_{\gamma }\mathrm{Ric}\left( \gamma ^{\prime },\gamma ^{\prime }\right)
\phi ^{2}\leq \left( n-1\right) \int_{\gamma }\left( \phi ^{\prime }\right)
^{2}
\end{equation*}%
for any function $\phi $ with compact support on $\gamma .$ By choosing
\begin{equation*}
\phi =\left\{
\begin{array}{c}
1 \\
t \\
r-t \\
0%
\end{array}%
\right.
\begin{array}{l}
\text{for }1\leq t\leq r-1 \\
\text{for \ }0\leq t\leq 1 \\
\text{for }r-1\leq t\leq r\text{\ } \\
\text{otherwise}%
\end{array}%
\end{equation*}%
and using that $\mathrm{Ric}\left( \gamma ^{\prime },\gamma ^{\prime }\right)
+f^{\prime \prime }\left( t\right) \geq \frac{1}{2},$ we get%
\begin{eqnarray}
\frac{1}{2}r-c_{1} &\leq &\int_{0}^{r}f^{\prime \prime }\left( t\right) \phi
^{2}\left( t\right) dt  \label{f1} \\
&=&-2\int_{0}^{r}f^{\prime }\left( t\right) \phi \left( t\right) \phi
^{\prime }\left( t\right) dt.  \notag
\end{eqnarray}%
Here $c_{1}$ depends
only on $n$. Observe that for $x,y\in M$ with $d\left( x,y\right)
\leq 1,$ by (\ref{f0}),
\begin{equation*}
\left\vert 2\sqrt{f\left( x\right) }-2\sqrt{f\left( y\right) }\right\vert
\leq d\left( x,y\right) \leq 1.
\end{equation*}%
So
\begin{equation*}
\left\vert \int_{0}^{1}f^{\prime }\left( t\right) \phi \left( t\right)
dt\right\vert \leq \int_{0}^{1}\sqrt{f}\left( t\right) \phi \left( t\right)
dt\leq c_{2}
\end{equation*}%
for a constant $c_{2}$ depending on $f\left( p\right) .$ Using this,
we conclude from (\ref{f1}) that%
\begin{eqnarray*}
\frac{1}{2}r-c &\leq &2\int_{r-1}^{r}f^{\prime }\left( t\right) \phi \left(
t\right) dt\leq 2\int_{r-1}^{r}\sqrt{f\left( t\right) }\phi \left( t\right)
dt \\
&\leq &2\sqrt{f\left( r\right) }\int_{r-1}^{r}\phi \left( t\right) dt+c \\
&=&\sqrt{f\left( r\right) }+c,
\end{eqnarray*}%
where the constant $c$ depends on $n$ and $f\left( p\right).$ This proves the
lower bound estimate.
\end{proof}

\section{\label{VL}Volume lower bound}

In this section, our focus is on proving the second part of Theorem \ref{Vol} on
the volume lower bound. We will use a similar
strategy as in the shrinking Ricci soliton case \cite{MW1}.
We start by making some preparations.

\subsection{Preliminary Estimates}

The following proposition follows from (\ref{vu11}). But we give a short proof
here.

\begin{proposition}
\label{Relative} Let $\left( M,g,e^{-f}dv\right) $ be a complete smooth
metric measure space of dimension $n$. Assume that $\mathrm{Ric}_{f}\geq
\frac{1}{2}$ and $\left\vert \nabla f\right\vert ^{2}\leq f.$ Then there exists $%
c>0$, depending only on $n$ and $f\left( p\right) ,$ such that for any $r>1$
and $0\leq s\leq 1,$%
\begin{equation*}
\mathrm{Area}\left( \partial B_{p}\left( r+s\right) \right) \leq c\mathrm{%
Area}\left( \partial B_{p}\left( r\right) \right) .
\end{equation*}
\end{proposition}

\begin{proof}[Proof of Proposition \protect\ref{Relative}]
We follow the notations in Section \ref{VU}. Recall from (\ref{vu3}) that
along a minimizing geodesic $\gamma $ with $\gamma \left( 0\right) =p$ and $%
\gamma ^{\prime }\left( 0\right) =\xi,$
\begin{equation*}
m\left( t\right) \leq \frac{n-1}{t}-\frac{1}{6}t+f^{\prime }\left( t\right) -%
\frac{2}{t^{2}}\int_{0}^{t}sf^{\prime }\left( s\right) ds,
\end{equation*}%
where $f\left( t\right) :=f\left( \gamma \left( t\right) \right) $ and $%
m\left( t\right) :=\frac{d}{dt}\ln J\left( p,t,\xi \right).$  After integrating the last term by parts, it follows
\begin{equation*}
m\left( t\right) \leq \frac{n-1}{t}-\frac{1}{6}t+f^{\prime }\left( t\right) -%
\frac{2}{t}f\left( t\right) +\frac{2}{t^{2}}\int_{0}^{t}f\left( s\right) ds.
\end{equation*}%
Using Proposition \ref{f}, one sees that
\begin{equation*}
m\left( t\right) \leq c
\end{equation*}%
for some constant $c>0$ depending only on $n$ and $f\left( p\right) .$
Integrating this inequality from $t=r$ to $t=r+s$, we get%
\begin{equation*}
J\left( p,r+s,\xi \right) \leq J\left( p,r,\xi \right) e^{cs}\leq
e^{c}J\left( p,r,\xi \right) .
\end{equation*}
Finally, integrating with respect to $\xi$ on $S_{p}M$ yields the result.
\end{proof}

We will also need a volume decay estimate for geodesic balls of unit radius.
Such a result was first established for shrinking Ricci solitons in \cite{MW1}.
Here, we extend it to our setting with an improved version. This improved estimate is crucial for the proof of Theorem \ref{Vol}.

\begin{lemma}
\label{Decay} Let $\left( M,g,e^{-f}dv\right) $ be a complete smooth metric
measure space of dimension $n$. Assume that $\mathrm{Ric}_{f}\geq \frac{1}{2}
$ and $\left\vert \nabla f\right\vert ^{2}\leq f.$ Then there exists a
constant $C_{0}>0$ so that
\begin{equation*}
\mathrm{Vol}\left( B_{x}\left( 1\right) \right) \geq \exp \left( -C_{0}\sqrt{%
R\ln R}\right) \mathrm{Vol}\left( B_{p}\left( 1\right) \right)
\end{equation*}%
for $R:=d\left( p,x\right)>2.$ The constant $%
C_{0}$ depends only on $n$ and $f\left( p\right) .$
\end{lemma}

\begin{proof}[Proof of Lemma \protect\ref{Decay}]
Let us denote the volume form in
geodesic coordinates centered at $x$ by
\begin{equation*}
dV|_{\exp _{x}\left( r\xi \right) }=J\left( x,r,\xi \right) drd\xi
\end{equation*}%
for $r>0$ and $\xi \in S_{x}M,$ the unit tangent sphere at $x$. Let $R:=d\left(
p,x\right).$ Let $\gamma \left( s\right) $ be a minimizing normal geodesic with $\gamma \left( 0\right) =x$ and $\gamma \left( T\right)
\in B_{p}\left( 1\right) $ for some $T>0.$ By the triangle inequality, we know
that
\begin{equation}
R-1\leq T\leq R+1.  \label{d0}
\end{equation}%
Along $\gamma $, according to (\ref{vu1}),%
\begin{equation*}
m^{\prime }\left( r\right) +\frac{1}{n-1}m^{2}\left( r\right) \leq -\frac{1}{%
2}+f^{\prime \prime }\left( r\right) ,
\end{equation*}%
where $m\left( r\right) :=\frac{d}{dr}\ln J\left( x,r,\xi \right) .$

For an arbitrary $k\geq 2$, multiplying this by $r^{k}$ and integrating from $%
r=0$ to $r=t$, we have%
\begin{gather}
\int_{0}^{t}m^{\prime }\left( r\right) r^{k}dr+\frac{1}{n-1}%
\int_{0}^{t}m^{2}\left( r\right) r^{k}dr\leq -\frac{1}{2\left( k+1\right) }%
t^{k+1}  \label{d0'} \\
+\int_{0}^{t}f^{\prime \prime }\left( r\right) r^{k}dr.  \notag
\end{gather}%
After integrating the first term in (\ref{d0'}) by parts and rearranging terms,
we get%
\begin{eqnarray*}
&&m\left( t\right) t^{k}+\frac{1}{n-1}\int_{0}^{t}\left( m\left( r\right) r^{%
\frac{k}{2}}-\left( n-1\right) \frac{k}{2}r^{\frac{k}{2}-1}\right) ^{2}dr \\
&\leq &\frac{\left( n-1\right) k^{2}}{4\left( k-1\right) }t^{k-1}-\frac{1}{%
2\left( k+1\right) }t^{k+1}+\int_{0}^{t}f^{\prime \prime }\left( r\right)
r^{k}dr.
\end{eqnarray*}%
In particular,
\begin{equation*}
m\left( t\right) \leq \frac{\left( n-1\right) k^{2}}{4\left( k-1\right) }%
\frac{1}{t}-\frac{1}{2\left( k+1\right) }t+\frac{1}{t^{k}}%
\int_{0}^{t}f^{\prime \prime }\left( r\right) r^{k}dr.
\end{equation*}%
Integrating this from $t=1$ to $t=T$, we obtain for some constant $c$
depending only on $n$, %
\begin{equation}
\ln \frac{J\left( x,T,\xi \right) }{J\left( x,1,\xi \right) }\leq c\,k\,\ln T-%
\frac{1}{4\left( k+1\right) }\,T^{2}+A,  \label{d1}
\end{equation}%
where the term $A$ is given by
\begin{equation*}
A:=\int_{1}^{T}\frac{1}{t^{k}}\int_{0}^{t}f^{\prime \prime }\left( r\right)
r^{k}dr.
\end{equation*}%

We now estimate $A$ in the right side of (\ref{d1}). Integrating by parts
implies%

\begin{equation}
A=f\left( T\right) -f\left( 1\right) -k\int_{1}^{T}\frac{1}{t^{k}}%
\int_{0}^{t}f^{\prime }\left( r\right) r^{k-1}drdt.  \label{d2}
\end{equation}%
Noting that by integrating by parts,%
\begin{eqnarray*}
&&-k\int_{1}^{T}\frac{1}{t^{k}}\int_{0}^{t}f^{\prime }\left( r\right)
r^{k-1}drdt=-\frac{k}{k-1}\left( f\left( T\right) -f\left( 1\right) \right)
\\
&&+\frac{k}{k-1}\frac{1}{t^{k-1}}\left( \int_{0}^{t}f^{\prime }\left(
r\right) r^{k-1}dr\right) |_{t=1}^{t=T}.
\end{eqnarray*}%
Putting this into (\ref{d2}), we conclude%
\begin{eqnarray}
A &=&-\frac{1}{k-1}\left( f\left( T\right) -f\left( 1\right) \right)
\label{d4} \\
&&+\frac{k}{k-1}\frac{1}{t^{k-1}}\left( \int_{0}^{t}f^{\prime }\left(
r\right) r^{k-1}dr\right) |_{t=1}^{t=T}.  \notag
\end{eqnarray}%
We now choose
\begin{equation}
k:=\sqrt{\frac{T}{\ln T}}.  \label{d3}
\end{equation}%
Using the hypothesis on $f$ and Proposition \ref{f}, we find that
\begin{eqnarray*}
\left\vert f^{\prime }\left( r\right) \right\vert &\leq &\sqrt{f}\left(
\gamma \left( r\right) \right) \leq \frac{1}{2}d\left( p,\gamma \left(
r\right) \right) +c \\
&\leq &\frac{1}{2}\left( T-r\right) +2c.
\end{eqnarray*}%
Similarly,
\begin{equation*}
f\left( r\right) \geq \frac{1}{4}\left( T-r\right) ^{2}-c\left( T-r\right)-c
\end{equation*}
and
\begin{equation*}
f\left( r\right) \leq \frac{1}{4}\left( T-r\right) ^{2}+c\left(
T-r\right) +c
\end{equation*}%
for $0\leq r\leq T.$ Everywhere in this proof, $c$ denotes
a constant depending only on $n$ and $f\left( p\right) .$ We now begin to
estimate the terms in (\ref{d4}). Integrating by parts, we have%
\begin{gather}
\frac{k}{k-1}\frac{1}{T^{k-1}}\int_{0}^{T}f^{\prime }\left( r\right)
r^{k-1}dr=\frac{k}{k-1}f\left( T\right) -\frac{k}{T^{k-1}}%
\int_{0}^{T}f\left( r\right) r^{k-2}dr  \label{d5} \\
\leq c-\frac{k}{4T^{k-1}}\int_{0}^{T}\left( T-r\right) ^{2}r^{k-2}dr+\frac{ck%
}{T^{k-1}}\int_{0}^{T}\left( T-r\right) r^{k-2}dr  \notag \\
\leq \frac{c}{k}T-\frac{1}{2\left( k-1\right) \left( k+1\right) }T^{2}.
\notag
\end{gather}%
Moreover,%
\begin{eqnarray}
\frac{k}{k-1}\int_{0}^{1}f^{\prime }\left( r\right) r^{k-1}dr &\leq &\frac{k%
}{k-1}\int_{0}^{1}\left( \frac{1}{2}\left( T-r\right) +2c\right) r^{k-1}dr
\label{d6} \\
&\leq &\frac{c}{k}T.  \notag
\end{eqnarray}%
\newline
Also,
\begin{equation*}
-\frac{1}{k-1}\left( f\left( T\right) -f\left( 1\right) \right) \leq \frac{1%
}{4\left( k-1\right) }T^{2}+\frac{c}{k}T.
\end{equation*}%
Together this with (\ref{d6}) and (\ref{d5}), we conclude from (\ref{d4}) that
\begin{equation*}
A\leq \frac{1}{4\left( k+1\right) }T^{2}+\frac{c}{k}T.
\end{equation*}%
Now plugging the above estimate for $A$ into (\ref{d1}) implies
\begin{equation*}
\ln \frac{J\left( x,T,\xi \right) }{J\left( x,1,\xi \right) }\leq ck\ln T+%
\frac{c}{k}T\leq c\sqrt{T\ln T},
\end{equation*}%
where the last inequality follows by (\ref{d3}). We have thus proved that
\begin{equation*}
J\left( x,1,\xi \right) \geq \exp \left( -c\sqrt{R\ln R}\right) J\left(
x,T,\xi \right) .
\end{equation*}%
By integrating this over a subset of $S_{x}M$ consisting of all unit tangent
vectors $\xi $ so that $\exp _{x}\left( T\xi \right) \in B_{p}\left(
1\right) $ for some $T,$ it follows that
\begin{equation*}
\mathrm{Area}\left( \partial B_{x}\left( 1\right) \right) \geq \exp \left( -c%
\sqrt{R\ln R}\right) \mathrm{Vol}\left( B_{p}\left( 1\right) \right) ,
\end{equation*}%
where $R=d\left( p,x\right) .$ Clearly, for $\frac{1}{2}\leq t\leq 1,$
a similar estimate holds for
$\mathrm{Area}\left( \partial B_{x}\left( t\right) \right).$ Therefore,
\begin{equation*}
\mathrm{Vol}\left( B_{x}\left( 1\right) \right) \geq \exp \left( -c\sqrt{%
R\ln R}\right) \mathrm{Vol}\left( B_{p}\left( 1\right) \right) .
\end{equation*}%
This proves the result.
\end{proof}

\subsection{Volume lower bound}

We are now ready to prove the second part of Theorem \ref{Vol}. It
seems that the standard techniques in comparison geometry are not powerful
enough to prove the linear growth volume lower bound. Inspired by Perelman's work in
\cite{P}, we utilize the Log-Sobolev inequality in \cite{BE}, which says that
\begin{equation*}
\int_{M}\phi ^{2}\ln \phi ^{2}e^{-f}\leq 4\int_{M}\left\vert \nabla \phi
\right\vert ^{2}e^{-f}
\end{equation*}%
for any compactly supported function $\phi $ satisfying $\int_{M}\phi
^{2}e^{-f}=\int_{M}e^{-f}.$ A more useful form for us is
obtained by replacing
\begin{equation*}
\phi :=\left( \frac{\int_{M}e^{-f}}{\int_{M}u^{2}}\right) ^{\frac{1}{2}}\,u\,e^{%
\frac{f}{2}}
\end{equation*}%
for any smooth function $u$ with compact support. Then a direct computation leads
to an inequality of the form%
\begin{gather}
\int_{M}u^{2}\ln u^{2}-\left( \int_{M}u^{2}\right) \ln \left(
\int_{M}u^{2}\right) \leq 4\int_{M}\left\vert \nabla u\right\vert
^{2}-\left( \ln \mu \right) \int_{M}u^{2}  \label{log_sob} \\
+\int_{M}\left( \left\vert \nabla f\right\vert ^{2}-f\right)
u^{2}+2\int_{M}\left\langle \nabla f,\nabla u^{2}\right\rangle   \notag
\end{gather}%
for any function $u$ with compact support in $M.$
Here, we have denoted $\mu :=\int_{M}e^{-f}dv<\infty.$
In the case of gradient
shrinking Ricci solitons, the inequality can be further simplified, see
\cite{CN}.

\begin{theorem}
\label{Vol_Low}Let $\left( M,g,f\right) $ be a complete smooth metric
measure space of dimension $n$. Assume that $\mathrm{Ric}_{f}\geq \frac{1}{2}
$ and $\left\vert \nabla f\right\vert ^{2}\leq f.$ Then there exists a
constant $c_{0}>0,$ depending only on $n,\mu $ and $f\left( p\right) ,$ so
that%
\begin{equation*}
\mathrm{Area}\left( \partial B_{p}\left( t\right) \right) \geq c_{0},
\end{equation*}%
for any $t>1.$ In particular,
\begin{equation*}
\mathrm{Vol}\left( B_{p}\left( R\right) \right) \geq c_{0}R,
\end{equation*}%
for any $R>1.$
\end{theorem}

\begin{proof}[Proof of Theorem \protect\ref{Vol_Low}]
Since $\left\vert \nabla f\right\vert ^{2}\leq f,$ it follows from (\ref%
{log_sob}) that%
\begin{gather}
\int_{M}u^{2}\ln u^{2}-\left( \int_{M}u^{2}\right) \ln \left(
\int_{M}u^{2}\right) \leq 4\int_{M}\left\vert \nabla u\right\vert
^{2}+C\int_{M}u^{2}  \label{ls} \\
+2\int_{M}\left\langle \nabla f,\nabla u^{2}\right\rangle .  \notag
\end{gather}%
In \cite{MW1}, we proved the volume lower bound in the particular case of
gradient shrinking Ricci solitons. While the proof here uses some of the
ingredients from \cite{MW1}, one significant difference is the last term in (%
\ref{ls}) which has a different coefficient in the shrinking solitons case.
This turns out to cause some technical difficulties and requires the use of
Lemma \ref{Decay}. We also take this opportunity to present a more
streamlined argument than that in \cite{MW1} by avoiding the discussion
of whether the total volume is finite or not.

For simplicity, we use the notations
\begin{equation*}
B\left( t\right) :=B_{p}\left( t\right) \text{ \ and }V\left( t\right) :=%
\mathrm{Vol}\left( B_{p}\left( t\right) \right) .
\end{equation*}%
Define a function $u_{t}:M\rightarrow \mathbb{R}$ by
\begin{equation*}
u_{t}\left( x\right) =\left\{
\begin{array}{c}
t+1-r\left( x\right) \\
r\left( x\right) -\left( t-1\right) \\
0%
\end{array}%
\right.
\begin{array}{l}
\text{on }B\left( t+1\right) \backslash B\left( t\right) \\
\text{on }B\left( t\right) \backslash B\left( t-1\right) \\
\text{otherwise}%
\end{array}%
\end{equation*}%
 As before, $r(x):=d(p,x)$. Plugging $u_{t}$ into (\ref{ls}) and noting that $x\ln x\geq -\frac{1}{e}$
for any $x>0,$ we obtain%
\begin{gather}
-\int_{M}u_{t}^{2}\ln \left( \int_{M}u_{t}^{2}\right) \leq C\left( V\left(
t+1\right) -V\left( t-1\right) \right)  \label{vs1} \\
+2\int_{B\left( t+1\right) \backslash B\left( t-1\right) }\left\langle
\nabla f,\nabla u_{t}^{2}\right\rangle .  \notag
\end{gather}%
Everywhere in this proof, $C$ denotes a constant depending on $n,\mu $
and $f\left( p\right) $. Let us denote
\begin{equation*}
y\left( t\right) :=\int_{M}u_{t}^{2}.
\end{equation*}%
We compute
\begin{eqnarray}
\int_{B\left( t+1\right) \backslash B\left( t-1\right) }\left\langle \nabla
f,\nabla u_{t}^{2}\right\rangle &=&2\int_{B\left( t\right) \backslash
B\left( t-1\right) }\left\langle \nabla f,\nabla r\right\rangle u_{t}
\label{vs2} \\
&&-2\int_{B\left( t+1\right) \backslash B\left( t\right) }\left\langle
\nabla f,\nabla r\right\rangle u_{t}.  \notag
\end{eqnarray}%
We now estimate each term in (\ref{vs2}). Using Proposition \ref{f}, we get
\begin{gather}
2\int_{B\left( t\right) \backslash B\left( t-1\right) }\left\langle \nabla
f,\nabla r\right\rangle u_{t}\leq 2\int_{B\left( t\right) \backslash B\left(
t-1\right) }\left\vert \nabla f\right\vert u_{t}  \label{vs3} \\
\leq t\int_{B\left( t\right) \backslash B\left( t-1\right) }u_{t}+C\left(
V\left( t\right) -V\left( t-1\right) \right) .  \notag
\end{gather}

By (\ref{vu3}), we have
\begin{equation}
\left\langle \nabla f,\nabla r\right\rangle \geq \Delta r-\frac{n-1}{r}+%
\frac{1}{6}r+\frac{2}{r^{2}}\int_{0}^{r}tf^{\prime }\left( t\right) dt.
\label{vs4}
\end{equation}%
We integrate the last term by parts and use Proposition \ref{f} to conclude%
\begin{eqnarray*}
\frac{2}{r^{2}}\int_{0}^{r}tf^{\prime }\left( t\right) dt &=&\frac{2}{r}%
f\left( r\right) -\frac{2}{r^{2}}\int_{0}^{r}f\left( t\right) dt \\
&\geq &\frac{1}{3}r-C.
\end{eqnarray*}%
Therefore, (\ref{vs4}) implies%
\begin{equation*}
\left\langle \nabla f,\nabla r\right\rangle \geq \Delta r+\frac{1}{2}r-C.
\end{equation*}%
Applying this to the second term in (\ref{vs2}), we estimate%
\begin{gather*}
-2\int_{B\left( t+1\right) \backslash B\left( t\right) }\left\langle \nabla
f,\nabla r\right\rangle u_{t}\leq -2\int_{B\left( t+1\right) \backslash
B\left( t\right) }\left( \Delta r+\frac{1}{2}r-C\right) u_{t} \\
\leq 2\int_{B\left( t+1\right) \backslash B\left( t\right) }\left\langle
\nabla r,\nabla u_{t}\right\rangle +2A\left( t\right) -\int_{B\left(
t+1\right) \backslash B\left( t\right) }ru_{t} \\
+C\left( V\left( t+1\right) -V\left( t\right) \right) .
\end{gather*}%
Notice that $\left\langle \nabla r,\nabla u\right\rangle \leq 0$ on $B\left(
t+1\right) \backslash B\left( t\right).$ Also, by the mean value theorem,
\begin{equation*}
V\left( t\right) -V\left( t-1\right) =A\left( \xi \right) \geq cA\left(
t\right)
\end{equation*}%
for some $t-1\leq \xi \leq t,$ where we have used Proposition %
\ref{Relative}. Consequently,%
\begin{gather}
-2\int_{B\left( t+1\right) \backslash B\left( t\right) }\left\langle \nabla
f,\nabla r\right\rangle u_{t}\leq -t\int_{B\left( t+1\right) \backslash
B\left( t\right) }u_{t}  \label{vs5} \\
+C\left( V\left( t+1\right) -V\left( t-1\right) \right) .  \notag
\end{gather}%
Plugging (\ref{vs3}) and (\ref{vs5}) into (\ref{vs2}), we have
\begin{gather}
2\int_{B\left( t+1\right) \backslash B\left( t-1\right) }\left\langle \nabla
f,\nabla u_{t}^{2}\right\rangle \leq -2t\left( \int_{B\left( t+1\right)
\backslash B\left( t\right) }u_{t}-\int_{B\left( t\right) \backslash B\left(
t-1\right) }u_{t}\right)  \label{vs6} \\
+C\left( V\left( t+1\right) -V\left( t-1\right) \right) .  \notag
\end{gather}%
On the other hand, by a direct computation, we get
\begin{equation*}
\frac{d}{dt}y\left( t\right) =\frac{d}{dt}\int_{M}u_{t}^{2}=2\left(
\int_{B\left( t+1\right) \backslash B\left( t\right) }u_{t}-\int_{B\left(
t\right) \backslash B\left( t-1\right) }u_{t}\right) .
\end{equation*}%
Thus, (\ref{vs6}) becomes%
\begin{equation*}
2\int_{B\left( t+1\right) \backslash B\left( t-1\right) }\left\langle \nabla
f,\nabla u_{t}^{2}\right\rangle \leq -ty^{\prime }\left( t\right) +C\left(
V\left( t+1\right) -V\left( t-1\right) \right) .
\end{equation*}%
Feeding this back into (\ref{vs1}), we conclude%
\begin{equation}
ty^{\prime }\left( t\right) -y\left( t\right) \ln y\left( t\right) \leq
C\left( V\left( t+1\right) -V\left( t-1\right) \right) .  \label{vs7}
\end{equation}

By the mean value theorem, there exist $t-1\leq \xi _{1}\leq t+1$ and $%
t-2\leq \xi _{2}\leq t-1$ so that
\begin{gather*}
V\left( t+1\right) -V\left( t-1\right) =2A\left( \xi _{1}\right) \text{ \ and%
} \\
y\left( t-2\right) \geq \int_{B\left( t-1\right) \backslash B\left(
t-2\right) }u_{t-2}^{2}=\frac{1}{3}A\left( \xi _{2}\right) .
\end{gather*}%
By Proposition \ref{Relative},
\begin{equation*}
V\left( t+1\right) -V\left( t-1\right) \leq C\,y\left( t-2\right) .
\end{equation*}%
In conclusion, by (\ref{vs7}), we get
\begin{equation}
ty^{\prime }\left( t\right) -y\left( t\right) \ln y\left( t\right) \leq
C\,y\left( t-2\right) .  \label{vs8}
\end{equation}%
In \cite{MW1}, we showed that if $y$ satisfies such a differential inequality
and
\begin{equation*}
\liminf_{t\rightarrow \infty }y\left( t\right) =0,
\end{equation*}%
then $y$ decays exponentially, i.e.,%
\begin{equation*}
y\left( t\right) \leq \frac{C}{e^{at}}
\end{equation*}%
for some $a>0.$ For the sake of completeness, we include the details
below. In fact, we will prove a stronger statement. Let $\delta >0$ be
sufficiently small to be chosen later, which depends only on $n,\mu $ and $f\left(
p\right) $. Let us assume that there exists an $\frac{1}{2}>\varepsilon >0$
so that
\begin{equation}
y\left( \frac{1}{\varepsilon }\right) <\delta .  \label{vs9}
\end{equation}%
For $t_{0}:=\frac{1}{\varepsilon }\geq 2,$ we claim that
\begin{equation}
y\left( t\right) <\sqrt{\delta }e^{-\varepsilon t}\ \ \ \ \text{for any }%
t_{0}+2\leq t\leq t_{0}+4.  \label{vs10}
\end{equation}

Observe first that by the mean value theorem, we have
\begin{eqnarray}
\delta &>&y\left( t_{0}\right) =\int_{t_{0}-1}^{t_{0}+1}u_{t_{0}}^{2}\left(
r\right) A\left( r\right) dr  \label{vs11} \\
&=&A\left( \xi _{1}\right) \int_{t_{0}-1}^{t_{0}+1}u_{t_{0}}^{2}\left(
r\right) dr=\frac{2}{3}A\left( \xi _{1}\right) ,  \notag
\end{eqnarray}%
where $t_{0}-1\leq \xi _{1}\leq t_{0}+1.$ For $t\geq t_{0}+2,$
by the mean value theorem,
\begin{equation*}
y\left( t\right) =\int_{t-1}^{t+1}u_{t}^{2}\left( r\right) A\left( r\right)
dr=\frac{2}{3}A\left( \xi _{2}\right)
\end{equation*}%
for some $t-1\leq \xi _{2}\leq t+1.$ So by Proposition \ref{Relative}, %
\begin{equation*}
y\left( t\right) =\frac{2}{3}A\left( \xi _{2}\right) \leq CA\left( \xi
_{1}\right) \leq C\,\delta .
\end{equation*}%
Since $t_{0}\varepsilon =1,$ we have
\begin{equation}
y\left( t\right) \leq C\,\delta e^{-\varepsilon t}  \label{vs12}
\end{equation}%
for any $t_{0}+2\leq t\leq t_{0}+4.$ We now choose $\delta $
sufficiently small so that $C\,\delta <\sqrt{\delta }.$ Then, (\ref{vs12})
obviously implies (\ref{vs10}).

Now we claim that
\begin{equation}
y\left( t\right) <\sqrt{\delta }e^{-\varepsilon t},  \label{vs13}
\end{equation}%
for $t\geq t_{0}+2.$ If (\ref{vs13}) fails to be true for some $t\geq
t_{0}+2,$ then there exists a first $t=r$ so that $y\left( r\right) =\sqrt{\delta
}e^{-\varepsilon r}.$ The choice of $r$ implies
\begin{eqnarray*}
y\left( r\right) &=&\sqrt{\delta }e^{-\varepsilon r} \\
y^{\prime }\left( r\right) &\geq &-\varepsilon \sqrt{\delta }e^{-\varepsilon
r}.
\end{eqnarray*}

Since (\ref{vs10}) is true for $t\leq t_{0}+4,$ we know that $r-2\geq
t_{0}+2.$ Consequently,
\begin{equation*}
y\left( r-2\right) \leq \sqrt{\delta }e^{-\varepsilon \left( r-2\right) }.
\end{equation*}%
Applying (\ref{vs8}) to $t=r,$ one sees%
\begin{gather*}
-\varepsilon \sqrt{\delta }re^{-\varepsilon r}+\sqrt{\delta }e^{-\varepsilon
r}\left( -\frac{1}{2}\ln \delta +\varepsilon r\right) \leq ry^{\prime
}\left( r\right) -y\left( r\right) \ln y\left( r\right) \\
\leq Cy\left( r-2\right) \leq C\sqrt{\delta }e^{-\varepsilon \left(
r-2\right) }.
\end{gather*}%
After some simplification, this gives%
\begin{equation*}
-\ln \delta \leq 2e^{2\varepsilon }C\leq 5C.
\end{equation*}%
But this is impossible if $\delta $ is chosen sufficiently small, say, $\delta
=e^{-6C}$. Therefore, (\ref{vs13}) is true for all $t\geq t_{0}.$ On the other
hand, this contradicts with Lemma \ref{Decay}. The contradiction implies that there exists $\delta >0$ depending only on $n,\mu $ and $f\left(
p\right) $ so that
\begin{equation*}
y\left( t\right) \geq \delta \text{ \ for any }t\geq 2.
\end{equation*}%
By the mean value theorem, we have $\mathrm{Area}\left( B_{p}\left( \xi
\right) \right) \geq \frac{1}{2}\delta $ for some $t-1\leq \xi \leq t+1.$
By Proposition \ref{Relative}, this implies $\mathrm{Area}\left( B_{p}\left(
t-1\right) \right) \geq \frac{1}{C}\delta $ for any $t\geq 2.$ This proves
the result.
\end{proof}

We conclude this section with an example to show the sharpness of our
assumption on $f.$ The volume estimates are obviously sharp as this is the case
for shrinking Ricci solitons.

Let us consider $M=\mathbb{R}\times N,$ where $N$ is compact and
\begin{equation*}
ds_{M}^{2}=dt^{2}+e^{-2at}ds_{N}^{2}.
\end{equation*}%
Here $t\in \mathbb{R}$ and $a>0$ is small. The Ricci curvature of $M$ is
computed as
\begin{eqnarray*}
\mathrm{Ric}_{11} &=&-\left( n-1\right) a^{2}g_{11} \\
\mathrm{Ric}_{\alpha \beta } &=&\mathrm{Ric}_{\alpha \beta }^{N}-\left(
n-1\right) a^{2}g_{\alpha \beta },
\end{eqnarray*}%
where $e_{1}=\frac{\partial }{\partial t}$ and $\left\{ e_{\alpha }\right\}
_{\alpha \geq 2}$ is tangential to $N.$ For $f:=\left( \frac{1}{4}+\frac{%
\left( n-1\right) a^{2}}{2}\right) t^{2},$ we have
\begin{eqnarray*}
f_{11} &=&\left( \frac{1}{2}+\left( n-1\right) a\right) g_{11} \\
f_{\alpha \beta } &=&\left( \frac{1}{2}+\left( n-1\right) a^{2}\right)
\left( 1-at\right) g_{\alpha \beta }.
\end{eqnarray*}%
So $\mathrm{Ric}_{f}\geq \frac{1}{2}$ is true if for some $%
c_{0},c_{1}>0$ depending on $a,$
\begin{equation*}
\mathrm{Ric}_{\alpha \beta }^{N}\geq \left( c_{0}+c_{1}t\right) g_{\alpha
\beta }=\left( c_{0}+c_{1}t\right) e^{-2at}h_{\alpha \beta },
\end{equation*}%
where $h$ denotes the metric on $N$. Since $\left( c_{0}+c_{1}t\right)
e^{-2at}$ is bounded above for $t\in \mathbb{R},$ there exists such
a compact manifold $N$.
Note that for any $\varepsilon >0$ small, there exists $a>0$ so that
\begin{equation*}
\left\vert \nabla f\right\vert ^{2}\leq \left( 1+\varepsilon \right) f\text{
\ \ on }M.
\end{equation*}%
Moreover, $M$ has two ends, one with finite volume and another with exponential growth volume. Hence, the condition (\ref{C}) in Theorem \ref{Vol} can not be relaxed to $\left\vert \nabla f\right\vert ^{2}\leq \left( 1+\varepsilon \right) f.$

\vskip 0.3in

\address
{\noindent Department of Mathematics\\
University of Connecticut\\
Storrs, CT 06269\\
USA\\
\email{\textit{E-mail address}: {\tt ovidiu.munteanu@uconn.edu}
}

\vskip 0.3in

\address
{\noindent School of Mathematics \\ University of Minnesota\\ Minneapolis, MN
55455\\ USA\\ \email{\textit{E-mail address}: {\tt jiaping@math.umn.edu}}

\end{document}